\documentclass[12pt,a4paper]{article}

\usepackage[pagewise]{lineno}

\usepackage[utf8]{inputenc}
\usepackage[T1]{fontenc}
\usepackage[english]{babel}
\usepackage{amsmath,amsfonts,amssymb,amsthm}
\usepackage{latexsym}
\usepackage{amscd}
\usepackage{graphicx}
\usepackage{caption}
\usepackage{subcaption}
\usepackage{tikz-cd}
\usetikzlibrary{decorations.pathmorphing,decorations.pathreplacing}
\usepackage{url}
\usepackage{float}
\usepackage{geometry}
\allowdisplaybreaks

\newtheorem{theorem}{Theorem}[section]
\newtheorem{lemma}{Lemma}[section]
\newtheorem{definition}{Definition}[section]
\newtheorem{remark}{Remark}[section]
\newtheorem{example}{Example}[section]
\newtheorem{proposition}[theorem]{Proposition}
\newtheorem{corollary}[theorem]{Corollary}

\newenvironment{thm}{\begin{theorem}}{\end{theorem}}
\newenvironment{lem}{\begin{lemma}}{\end{lemma}}
\newenvironment{defi}{\begin{definition}}{\end{definition}}
\newenvironment{rem}{\begin{remark}}{\end{remark}}

\newenvironment{prop}{\begin{proposition}}{\end{proposition}}

\title{\textbf{Magnetically Insulated Diode: Existence of Solutions and Complex Bifurcation. I}}

\author{
\textbf{D.N. Sidorov}\\
\small Sino-Russian Joint Research Center for Advanced Energy and Power Systems,\\
\small Energy Systems Institute, Russian Academy of Science, Irkutsk, Russia\\
\small \texttt{dsidorov@isem.irk.ru}
\and
\textbf{A.V. Sinitsyn}\\
\small Departamento de Matemáticas, Universidad Nacional de Colombia,\\
\small Bogotá, Colombia\\
\small \texttt{asinitsyne@unal.edu.co}
\and
\textbf{O.D. Toledo Leguizamón}\\
\small Ingeniería de Sistemas y Computación, Universidad Nacional de Colombia,\\
\small Bogotá, Colombia\\
\small \texttt{otoledo@unal.edu.co}
\and
\textbf{L. Wang}\\
\small Department of Electrical Engineering and Automation,\\
\small Harbin Institute of Technology, Harbin, China\\
\small \texttt{wlg2001@hit.edu.cn}
}

\date{}

\begin{document}

\maketitle

\begin{abstract}

In order to avoid the electron oscillation of the cathode and enhance the work
efficiency of a vacuum diode, an approach for analyzing the solutions and complex bifurcation
has been proposed and used to determine the optimal trajectory of electron motion of the
vacuum diode. This work is focusing on the stationary self-consistent problem of magnetic insulation in a space-charge-limited vacuum diode, modeled by a singularly perturbed 1.5-dimensional Vlasov-Maxwell system. We focus on the insulated regime, characterized by the reflection of electrons back toward the cathode at a point 
$x^{*}.$
The analysis proceeds in two primary stages. First, the original Vlasov-Maxwell system is reduced to a nonlinear singular system of ordinary differential equations governing the electric and magnetic field potentials. Subsequently, this system is further reduced to a novel nonlinear singular ODE for an effective potential 
$\theta(x).$ 
 The existence of non-negative solutions to this final equation is established on the interval 
$[0, x^{*})$, where 
$\theta(x)>0$. This is achieved by reformulating the associated initial value problem into a system of coupled nonlinear Fredholm integral equations and proving the existence of fixed points for the corresponding operators.
The most significant and previously unexplored case occurs when 
$\theta(x)<0$ on the interval 
$(x^{*}, 1]$, which corresponds to the fully insulated diode. For this regime, we present a novel numerical analysis of complex solution bifurcations, examining their dependence on system parameters and boundary conditions. Bifurcation diagrams illustrating the solution 
$\theta(x)$ as a function of the free boundary 
$x^{*}$
  is constructed, and the insulated diode spacing is determined.
\end{abstract}


\newpage

\section {Introduction}

The reliable operation of modern electrical power grids and industrial systems is fundamentally dependent on advanced power electronic devices responsible for energy conversion and control \cite{book1}. Among these components, the vacuum diode plays a critical role in high-power applications, primarily functioning to enforce unidirectional current flow. A significant technological challenge, however, arises at extreme operational voltages, where the regulation of space-charge-limited electron flow becomes paramount to prevent operational instability and device failure.

A principal method for mitigating this issue is magnetic [7] insulation, a technique wherein an applied magnetostatic field is utilized to confine electron trajectories. This process effectively establishes a potential barrier that reflects electrons back toward the cathode, thereby inhibiting anodic current and enabling efficient diode operation at voltages that would otherwise lead to breakdown.

While the physical principle of magnetic insulation is established, the precise prediction of its onset and stability constitutes a non-trivial mathematical problem. The self-consistent interaction of charged particles with electromagnetic fields is governed by the Vlasov-Maxwell system [17], [10], [3], a set of kinetic equations whose complexity necessitates significant reduction for analytical and numerical treatment. This work addresses this challenge through the development and analysis of a simplified, yet physically representative, model derived from the singularly perturbed limit of the governing system.

Our investigation concentrates on the magnetically insulated diode (MID) regime, a previously underexplored scenario characterized by the consistent deflection of electrons. The primary objectives of this study are twofold:
1.  To establish a rigorous mathematical framework for the resulting nonlinear boundary value problem and prove the existence of physically admissible solutions, thereby verifying the feasibility of the insulated state.
2.  To conduct a comprehensive bifurcation analysis [12], [11], examining how the system's qualitative behavior evolves with key parameters. Through numerical simulation and the construction of bifurcation diagrams, we delineate the critical thresholds and parameter regions associated with the transition to magnetic insulation.

The implications of this research extend beyond theoretical interest. By providing a clarified mathematical description of the insulation phenomenon, this study contributes to the foundational knowledge required for designing next-generation power converters. The results offer a systematic framework for optimizing these components, with the potential to enhance their operational stability, energy density, and cost-effectiveness, thereby improving the reliability and performance of future power systems.

\section{Motivation}

The study of magnetically insulated diodes comes very natural in the area of high power vacuum electronics and plasma physics. When the magnetic field is strong enough, the electrons emitted at the cathode are not able to reach the anode and instead are deflected back, creating what is called the regime of magnetic insulation. This regime is very important because it decides if the current is only limited by space charge effects, or by the combination of electric and magnetic fields acting together.  

From the mathematical point of view, the description of this situation leads to a singularly perturbed Vlasov--Maxwell system in 1.5 dimensions, which in the singular limit is reduced to a nonlinear system of ODEs for the electric and magnetic potentials. A particular difficulty of this formulation is the presence of a free boundary $x^{*}$, the point where the effective potential changes sign. If $\Theta(x)$ remains positive in the whole interval, then the diode is non--insulated and all electrons reach the anode, while if $\Theta(x)$ becomes negative after some $x^{*}$ then the diode turns insulated and the electrons accumulate close to the cathode, forming a high energy layer.

\begin{figure}[h!]
\centering

\begin{subfigure}[b]{0.5\textwidth}
    \centering
    \includegraphics[height=3.5cm]{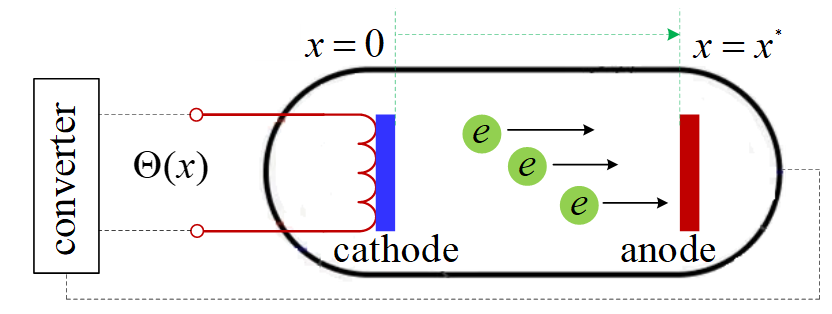}
    \caption{}
    \label{fig:left_image}
\end{subfigure}
\hfill
\begin{subfigure}[b]{0.45\textwidth}
    \centering
    \resizebox{!}{3.5cm}{
    \begin{tikzpicture}
        \draw[->] (0,0) -- (4,0) node[right] {$x$};
        \draw[->] (0,-2) -- (0,2) node[above] {$\Theta(x)$};

        \draw[dashed] (3,-1.5) -- (3,2);
        \node at (3.35,-0.1) [below] {$x^{*}$};

        \draw[thick] (0,1.5) .. controls (3,1.5) and (3,0) .. (3,0);
        \draw[thick] (0,-1.5) .. controls (3,-1.5) and (3,0) .. (3,0);

        \node at (0,0) [left] {$0$};
    \end{tikzpicture}
    }
    \caption{}
    \label{fig:right_tikz}
\end{subfigure}

\caption{
(a) Experimental configuration illustrating the magnetically insulated regime, where the cathode and anode are separated by a magnetic field that restricts electron motion. 
(b) Corresponding schematic of the effective potential $\Theta(x)$. 
Electrons emitted from the cathode at $x=0$ move within the region where $\Theta(x)>0$, reaching the turning point $x^{*}$ where $\Theta(x^{*})=0$ before being reflected back. 
For $x>x^{*}$, the potential becomes negative ($\Theta(x)<0$), representing a forbidden region where no electron trajectories can exist.
}
\label{fig:combined_figure}
\end{figure}
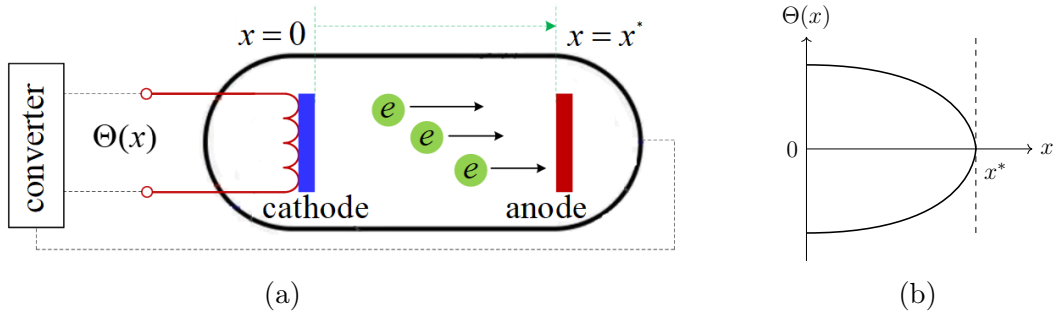

The physical scenario of magnetic insulation is schematically illustrated in Figure~\ref{fig:combined_figure}. The effective potential $\Theta(x)$, which combines both electric and magnetic field contributions, determines the electron dynamics. In the insulated regime shown, $\Theta(x)$ starts at zero at the cathode, rises to a maximum value $\Theta_1$, and then decreases, crossing zero at the free boundary point $x^{}$. Electrons emitted from the cathode can only propagate in the region where $\Theta(x)\geq 0$ ($0\leq x\leq x^{}$), reaching their turning point at $x^{}$ where $\Theta(x^{})=0$, and are subsequently reflected back toward the cathode. The region $x>x^{*}$, where $\Theta(x)<0$, is inaccessible to electrons, creating the characteristic electron layer near the cathode that defines the magnetically insulated state.

The symmetry of the problem and the fact that previous works already gave important reductions of the Vlasov--Maxwell system, inspired us to look at the model from a different angle. This new vision shows that even if the equations are very complex, there are structures hidden in them that can be used to simplify or to find equivalent formulations. At the same time, the complexity of the system becomes a valuable challenge, because it forces the use of modern computational tools. In this sense, numerical simulations and the construction of graphical displays, such as bifurcation diagrams, are not just complementary but necessary in order to support and reinforce the analytical results obtained.  

Although this configuration is clearly relevant from the physical point of view, the theory around magnetically insulated diodes is still far from complete. Questions of existence of solutions, the nature of bifurcations, or the determination of critical current values, are still open. Addressing these problems does not only give new insight into the qualitative behavior of the diode, but also provides a base for future applications in transport of energy, confinement of particles, and the development of high power electronic devices.
This article is an extended and refined version of 
our paper \cite{paper1}.

The remainder of this paper is structured as follows. In \textit{Setting of the problem and derivation of system (I)}, we establish the relativistic Vlasov–Maxwell model for the diode, perform the nondimensional rescaling, and derive the nonlinear limit system governing the coupled electrostatic and magnetic potentials. In \textit{The Cauchy problem for the limit system (I)}, we reformulate the equations using the substitution $(u,v)$, carefully define the notion of solution for the singular initial value problem, and prove the existence of nonnegative solutions by means of contraction mappings and integral equation techniques. In \textit{The Isolated Case}, we focus on the magnetically insulated regime $\Theta < 0$, derive the effective cubic equation governing the diode potential, and classify the bifurcation structure through discriminant analysis and explicit algebraic solutions. The section also explores different parameter regimes, separating algebraically possible solutions from physically admissible ones. Finally, in \textit{Conclusions}, we summarize our results and emphasize the interplay between analytical derivations and computational bifurcation diagrams in describing the operating regimes of magnetically insulated diodes.

\section{Setting of the problem and derivation of system (I)}

We consider a plane diode consisting of two perfectly conducting electrodes, 
a cathode $(X=0)$ and anode $(X=L)$ supposed to be infinite planes, parallel to
$(Y,Z)$.

The electrons, with charge $-e$ and mass $m$, are emitted at the cathode 
and submitted to an applied electromagnetic field
$$
{\bf E}_{\rm ext}=E_{\rm ext}{\bf X},\;\;\;\;{\bf B}_{\rm ext}=B_{\rm ext}{\bf Z}
$$
such that $E_{\rm ext}\le 0$ and $B_{\rm ext}\ge 0$.

We shall assume that the electron distribution function $F$ does not depend on
$Y$ and that the flow is stationary and collisionless. The system is then 
described by the so called 1.5 dimensional VM model
\begin{align}
&V_{X}\frac{\partial F}{\partial X}+e\Biggl ( \frac{d\Phi}{d X}-V_{Y}\frac{d A}
{d X}\Biggr )\frac{\partial F}{\partial P_{X}}+eV_{X}\frac{d A}{d X}\frac{\partial F}
{\partial P_{Y}}=0,  \label{(2.1)}\\
& \frac{d^{2}\Phi}{d X^{2}}=\frac{e}{\epsilon_{0}}N(X), \qquad X\in (0,L),
\label{(2.2)}\\
&\frac{d^{2}A}{d X^{2}}=-\mu_{0}J_{Y}(X), \qquad X\in(0,L) \label{(2.3)}
\end{align}
subject to the following boundary conditions:
\begin{align}
&F(0,P_{X},P_{Y})=G(P_{X},P_{Y}), \;\;P_{X}>0,  \label{(2.4)}\\
&F(L,P_{X},P_{Y})=0, \;\;\;\;\;\;\;\;\;\;\;\;\;\;\;\;P_{X}<0, \label{(2.5)}\\
&\Phi(0)=0,\;\;\Phi(L)=\Phi_{L}=-LE_{ext}, \label{(2.6)}\\
&A(0)=0,\;\;A(L)=A_{L}=L B_{ext}, \label{(2.7)}
\end{align}
where formulas \eqref{(2.4)} and \eqref{(2.5)} describe the injection profile at the cathode and 
at the anode, respectively, $E=-d\Phi/d X$, $B=-d A/d X$. The relationship between 
momentum and velocity is then given by the relativistic relations
$$
{\bf V}({\bf P})=\frac{{\bf P}}{\gamma m}, \;\;\;
\gamma=\sqrt{1+\frac{|{\bf P}|^{2}}{m^{2}c^{2}}},
$$
$$
{\bf V}=(V_{X},V_{Y}), \;\;\;{\bf P}=(P_{X},P_{Y}), \;\;|{\bf P}|^{2}=
P_{X}^{2}+P_{Y}^{2},
$$
or
$$
{\bf V}({\bf P})=\nabla_{{\bf P}}E({\bf P}),
$$
where $E$ is the relativistic kinetic energy and $c$ is the speed of light.

In the system \eqref{(2.1)}--\eqref{(2.3)}, the macroscopic quantities, namely the particle 
density $N$; $X$ and $Y$ components of the current density $J_{X}$, $J_{Y}$,
are respectively given by the following formulas
\begin{align}
&N(X)=\int\limits_{R^{2}}F(X,P_{X},P_{Y})dP_{X}dP_{Y}, \label{(2.8)}\\
&J_{X}=-e\int\limits_{R^{2}}V_{X}({\bf P})F(X,P_{X},P_{Y})dP_{X}dP_{Y}, \label{(2.9)}\\
&J_{Y}(X)=-e\int\limits_{R^{2}}V_{Y}({\bf P})F(X,P_{X},P_{Y})dP_{X}dP_{Y}. \label{(2.10)}
\end{align}
Here, $\epsilon_{0}$ and $\mu_{0}$ are  the vacuum permitivity and 
permeability respectively. 

The 1.5 model describes two principal regimes. For a strong applied magnetic field,
electrons do not reach the anode and come back to the cathode leading to a 
vanishing $J_{X}$ component of current density. When the applied magnetic field is not strong enough to insulate 
the diode, $J_{X}$ does not vanish and our model can be viewed as an 
approximate of the Maxwell equations.

Similarly to \eqref{(2.8)}--\eqref{(2.10)}, we define the moments associated with the incoming particle 
distribution function by
\begin{align}
&N^{G}=\int\limits_{R^{2}_{+}}G(P_{X},P_{Y})dP_{X}dP_{Y},\label{(2.11)}\\
&J_{X}^{G}=-e\int\limits_{R_{+}^{2}}V_{X}({\bf P})G(P_{X},P_{Y})dP_{X}dP_{Y}, \label{(2.12)}\\
&J_{Y}^{G}=-e\int\limits_{R_{+}^{2}}V_{Y}({\bf P})G(P_{X},P_{Y})dP_{X}dP_{Y}, \label{(2.13)}\\
&T^{G}=\int\limits_{R_{+}^{2}}E({\bf P})G(P_{X},P_{Y})dP_{X}dP_{Y}, \label{(2.14)}
\end{align}
where $R^{2}_{+}=\{(P_{X},P_{Y})\in R^{2}, P_{X}>0\}$,  and the thermal 
emission velocity  is $V^{G}=\sqrt{\frac{T^{G}}{mN^{G}}}$. The quantities 
\eqref{(2.11)}--\eqref{(2.14)}, respectively define the incoming particle density, 
$X$ and $Y$ components of the incoming current density and incoming particle kinetic 
energy.

In order to get a better insight into the behavior of the diode, we write the model 
\eqref{(2.1)}--\eqref{(2.7)} in dimensionless variables as P. Degond and P.-A. Raviart \cite{Degond1991,Degond1992}.

Let the diode be controlled in the Child-Langmuir regime [9]. In such a situation, the 
thermal velocity $V_{G}$ is much smaller than the typical drift velocity 
supposed to be of the order of the speed of light $c$. Letting 
$\varepsilon=\frac{V_{G}}{c}$, we shall assume that
$$
f(0,p_{x},p_{y})=g^{\varepsilon}(p_{x},p_{y})=\frac{1}{\varepsilon^{3}}g(
\frac{p_{x}}{\varepsilon}, \frac{p_{y}}{\varepsilon}), \;\;p_{x}>0,
$$
where $g$ is a given profile. The dimensionless system reads
\begin{align*}
&v_{x}\frac{\partial f^{\varepsilon}}{\partial x}+\biggl (\frac{d\varphi^{\varepsilon}}
{d x}-v_{y}\frac{d a^{\varepsilon}}{d x}\biggr )\frac{\partial f^{\varepsilon}}
{\partial p_{x}}+v_{x}\frac{d a^{\varepsilon}}{d x}\frac{\partial f^{\varepsilon}}{
\partial p_{y}}=0, \quad
(x,p_{x},p_{y})\in(0,1)\times R^{2},\\
&\frac{d^{2}\varphi^{\varepsilon}}{d x^{2}}=n^{\varepsilon}(x), \;\;x\in(0,1),
\\
&\frac{d^{2}a^{\varepsilon}}{d x^{2}}=j_{y}^{\varepsilon}(x), \;\;x\in(0,1),
\\
&n^{\varepsilon}(x)=\int\limits_{R^{2}_{+}}f^{\varepsilon}(x,p_{x},p_{y})dp_{x}dp_{y},
 \\
&j^{\varepsilon}_{y}(x)=\int\limits_{R^{2}_{+}}v_{y}f^{\varepsilon}(x,p_{x},p_{y})
dp_{x}dp_{y}=
\int\limits_{R_{+}^{2}}\frac{p_{y}}{\sqrt{1+|p|^{2}}}f^{\varepsilon}(x,p_{x},p_{y})dp_{x}
dp_{y},  \\
&f^{\varepsilon}(0,p_{x},p_{y})=g^{\varepsilon}(p_{x},p_{y})=\frac{1}
{\varepsilon^{3}}g(\frac{p_{x}}{\varepsilon},\frac{p_{y}}{\varepsilon}), \quad
p_{x}>0,\\
&f^{\varepsilon}(1,p_{x},p_{y})=0, \quad p_{x}<0, \\
&\varphi^{\varepsilon}(0)=0, \quad \varphi^{\varepsilon}(1)=\varphi_{L},
\\
&a^{\varepsilon}(0)=0, \quad a^{\varepsilon}(1)=a_{L}.  
\end{align*}
To derive the limit model (I) at $\varepsilon\rightarrow 0$, we consider 
the various invariants of the problem. The following two quantities are 
constants of motion
\begin{align*}
W^{\varepsilon}(x,p)=&\Xi(p)-\varphi^{\varepsilon}(x)- {\rm the\;\;
electron\;\;energy},\\
{\cal P}^{\varepsilon}_{y}(x,p)=&p_{y}-a^{\varepsilon}(x)- {\rm the\;\;
canonical\;\; momentum},  
\end{align*}
which means that on each electron trajectory (in the phase space), the 
above quantities are constant. Let us denote $f$, $n$, $a$, $j$, $\varphi$
$\ldots$ the limit as $\epsilon$ tends to zero $f^{\varepsilon}$, 
$n^{\varepsilon}$, $\ldots$. Since, in the limit $\varepsilon=0$, electrons  
are injected with zero velocity, it is readily seen that the electron energy
$W$ and canonical momentum ${\cal P}_{y}$ simultaneously vanish. Consequently,
\begin{align*}
p_{y}(x)=&a(x), \\
(p_{x}(x))^{2}=&(1+\varphi(x))^{2}-1-(a(x))^{2}
\end{align*}
and the following identities hold:
\begin{align*}
v_{x}(x)=&\frac{p_{x}(x)}{\sqrt{1+{\bf p}^{2}(x)}}=\frac{p_{x}(x)}{1+\varphi(x)},\\
v_{y}(x)=&\frac{v_{y}(x)}{\sqrt{1+{\bf p}^{2}(x)}}=\frac{a(x)}{1+\varphi(x)}.
\end{align*}
Let us now define the \emph{effective potential} by
$$
\Theta(x)=(1+\varphi(x))^{2}-1-(a(x))^{2}. 
$$
Electrons do not enter the diode unless the effective potential $\Theta$
is non-negative in the vicinity of the cathode. Therefore, we always have
$\Theta'(0)\ge 0$. Let $\Theta_{L}$ be the value of $\Theta$ at the anode
$$
\Theta_{L}=(1+\varphi_{L})^{2}-1-a_{L}^{2}.  
$$
If $\Theta_{L}<0$, electrons cannot reach the anode $x=1$; they are 
reflected by the magnetic forces back to the cathode and the diode is said to 
be \emph{magnetically insulated}. If $\Theta$ is non-negative, then all electrons 
are reached the anode and the diode is said to be \emph{noninsulated}.

We assume that
$$
\forall x\in(0,1], \;\;\;\Theta(x)>0, \;\;\;\Theta(1)-\Theta(0)=\Theta_{L}>0.
$$
Since no electron is injected at the anode, $j^{-}_{x}$ vanishes.
Hence
$$
j_{x}=j^{+}_{x}=\int\limits_{R^{2}_{+}}v_{x}f(x,p_{x},p_{y})dp_{x}dp_{y}
$$
and the distribution function is that of a monokinetic beam issued from 
the cathode $x=0$ with vanishing initial velocity
$$
f(x,{\bf P})=n(x)\delta\biggl (p_{x}-\sqrt{\Theta(x)}\biggr )\delta(p_{y}-a(x)).
$$
Therefore
$$
n(x)=\frac{j_{x}}{v_{x}(x)}=j_{x}\frac{1+\varphi(x)}{\sqrt{\Theta(x)}}, \;\;\;
j_{y}(x)=n(x)v_{y}(x)=j_{x}\frac{a(x)}{\sqrt{\Theta(x)}}.
$$
Inserting these expressions into Poisson's and Ampere's equations 
\eqref{(2.2)} and \eqref{(2.3)} gives
\begin{equation*}
\begin{array}{l}
\frac{d^{2}\varphi}{dx^{2}}(x)=j_{x}\displaystyle\frac{1+\varphi(x)}{\sqrt{(1+\varphi(x))^{2}-
1-(a(x))^{2}}}, \;\;\;\varphi(0)=0, \;\;\varphi(1)=\varphi_{L},\\
\frac{d^{2}a}{dx^{2}}(x)=j_{x}\displaystyle\frac{a(x)}{\sqrt{(1+\varphi(x))^{2}-1-(a(x))^{2}}},\;\;\;
a(0)=0, \;\;\;a(1)=a_{L}.
\end{array}
\eqno {\rm (I)}
\end{equation*}
In system (I) the unknowns are the electrostatic potential $\varphi$,
the magnetic potential $a$ and the current $j_{x}$ (which does not depend 
on $x$).

\subsection*{Magnetically insulating diode}

 Because in this case $\Theta<0$, then  the effective potential is repulsive. Electrons emitted from the cathode with zero initial velocity cannot reach the anode. They are deflected back to the cathode at some (unknown in advance) point of diode $x^{*}$  so that
\begin{equation}\label{MID-1}
\begin{array}{ll}
\forall x\in [0,x^{*}],\quad & \Theta(x)\geq 0 \qquad{\rm and} \qquad n(x)>0,\\
\forall x\in (x^{*},1], \quad & \Theta(x)<0\qquad {\rm and} \qquad f(x,p_{x},p_{y})=0.  
\end{array}    
\end{equation}
Note that the equations for the electric and magnetic potentials depend on the location of the points at which the effective potential $\Theta$ vanishes. P. Degond with colleagues [1] considered a special case: the \emph{quasilaminar model}, where $\Theta$ vanishes only at the boundaries of the electron layer. In the quasilaminar case, electrons leave the cathode, reach point $x^{*}$ and then deviate to the cathode.
The quasilaminar model is described by a system of nonlinear equations as follows:

On the interval $(0, x^{*})$ the system is given by 
\begin{equation}\label{MID-2}
\begin{array}{ll}
\frac{d^{2}\varphi}{dx^{2}}(x)= j_{x}\displaystyle\frac{1+\varphi(x)}{\sqrt{(1+\varphi(x))^{2}-1- (a(x))^{2}}},& \quad \forall x\in (0, x^{*}), \\
\frac{d^{2}a}{dx^{2}}(x)= j_{x}\displaystyle\frac{a(x)}{\sqrt{(1+\varphi(x))^{2}-1- (a(x))^{2}}},& \quad \forall x\in (0, x^{*}),    
\end{array} 
\end{equation}
and on the interval $(x^{*}, 1)$ the model is described by the system
\begin{equation}\label{MID-3}
\begin{array}{ll}
&\displaystyle\frac{d^{2}\varphi}{dx^{2}}(x)= \displaystyle\frac{d^{2}a}{dx^{2}}(x)=0, \qquad \forall x\in (x^{*}, 1), \\
&\varphi(0)=0, \quad \varphi(1)=\varphi_{L}, \quad a(0)=0, \quad a(1)=a_{L},           
\end{array}                
\end{equation}
with the condition that functions $\varphi, a$ and their first derivatives are continuous at $x^{*}$. 

Inside the electron layer $(0,x^{*})$ system \eqref{MID-2} is identical to system \eqref{MID-3} for the case of a noninsulating diode.

We consider a general case when a free point (free boundary) $x^{*}$ can move from anode to cathode and back to the anode. In this case, the insulated diod described by the system (I) on the interval $[0, 1]$.  The open problem is to find the position of the  point $x^{*}$ depending on the boundary conditions at the anode. What is the distance from the cathode to $x^{*}$. This interval $0 \leq x^{*}\leq d$ can form a high-energy layer of electrons near the cathode, which can be transmitted over long distances: transfer of energy. Also, the electron layer near cathode  can act as a magnetic wall. 

We split the problem into two intervals - the first one is IVP on  $[0, x^{*})$ with $\Theta(x)\geq 0$ and second interval $(x^{*}, 1]$ with $\Theta(x)\leq 0$ (Fig. 1). 

There is also a third regime described in the papers [4], [19], [13]. It consists in the fact that electrons exist that do not reach the anode and are not deflected to the cathode.  Electron flow oscillated between magnetically insulated and non-insulated states with this bimodal behavior eventually settling into a steady-state with characteristics of insulated and non-insulated flows.

\section{The Cauchy problem for the limit system (I)}

The regime of ``insulated'' diode on the interval $[0, x^{*}), \Theta(x)>0$ is described by the following nonlinear two-point boundary value problem (I). Almost at this interval, the insulated mode coincides with the non-insulated one.
%
After the substitution $\varphi + 1=:u$, $a=:v$ the equations of magnetic insulation have the form
\begin{align*}
u'' =& j_{x}\frac{u}{\sqrt{u^{2} - 1 - v^{2}}}, \;\;\;\; u(0) = 1, \;\;\; u(1) = \varphi_{L} + 1 =: \alpha, \;\;  u'(0) = 0\\
v'' =& j_{x}\frac{v}{\sqrt{u^{2} - 1 - v^{2}}}, \;\;\;\; v(0) = 0, \;\; v(1) = a_{L}, \;\;  v'(0) = \beta>0.
\end{align*}

First, we consider the initial value problem with $u(0) = 1, u'(0) = 0, v(0) = 0, v'(0) = \beta$. The second step, we intend to investigate the boundary value problem with the mode of magnetic insulation.

We define the effective potential $\theta(u,v) =: u^{2} - 1 - v^{2}$. 

For $\Omega =  \{(u,v) \in R^{2}; \; \theta(u,v)\le 0\}$ the right hand side is not defined.  In particular, for a solution $(u,v)$, $\theta(u(0),v(0)) = 0$, 
so the right hand side is not defined for $x = 0$. Therefore the concept of a solution must be carefully defined. We give a definition of a solution for the initial value problem.

\begin{defi} A function $(u,v) = (u(x),v(x))$ is a solution of the initial value problem (IVP) 
\begin{equation}\label{IVP1}
\begin{array}{l}
u''(x) = j_{x}\displaystyle\frac{u}{\sqrt{u^{2} - 1 - v^{2}}},     \\
v''(x) = j_{x} \displaystyle \frac{v}{\sqrt{u^{2} - 1 - v^{2}}},  
\end{array}
\end{equation}
with the initial conditions
\begin{equation}\label{IVP1'}
\begin{array}{ll}
 u(0) = 1,& \qquad u'(0) = 0\in R  \\
 v(0) = 0,&  \qquad v'(0) = \beta \in R
\end{array}
\end{equation}
on the interval $[0,\varepsilon)$ if
\begin{itemize}
\item[a)]
 $u,v \in C^{1}[0,\varepsilon) \cap C^{2}(0,\varepsilon)$,
\item[b)] $\theta(u(x),v(x))>0$ for $x\in (0,\varepsilon)$ that is, $(u(x), v(x))\notin \Omega$ for $x\in (0,\varepsilon)$,
\item[c)] $(u,v)$ satisfies \eqref{IVP1} on $(0,\varepsilon)$,
\item[d)] the initial conditions hold in \eqref{IVP1'}.
\end{itemize}
\end{defi}
\begin{prop}\label{prop1} Let $(u,v)$ be a solution of the IVP \eqref{IVP1}, \eqref{IVP1'} on $[0,\varepsilon)$ and define $\theta =: u^{2} - 1 - v^{2}$. Then 
\begin{itemize}
\item[i)] $\theta(x) \in C^{1}[0,\varepsilon)\cap C^{2}(0,\varepsilon)$,

\item[ii)] $\theta(0) = \theta'(0) = 0$ and $\theta(x)>0$ on $(0,\varepsilon)$,
\item[iii)] $\theta$ satisfies the differential equation
\begin{align}\label{prop3.1eq2}
\theta'' =& j_{x}\left(6\sqrt{\theta} +\frac{2}{\sqrt{\theta}} - 4\gamma\right), \qquad \gamma =: -\frac{1}{2j_{x}}\beta^{2}.\nonumber\\
=& \frac{6j_{x}}{\sqrt{\theta}}\left(\frac{1}{3} - \frac{2}{3}\gamma \sqrt{\theta}  + \theta\right) \quad {\rm on}\;\; (0,\varepsilon).  
\end{align}
\end{itemize}
\end{prop}
\begin{proof}
 i) and ii) follow directly from a), b) and 
 $$
 \theta'(0) = 2u(0)u'(0) - 2v(0)v'(0) = 2\cdot 1\cdot 0 - 2\cdot 0\cdot \beta = 0
 $$ with d).
To show iii), we have
\begin{align}\label{prop3.1eq3}
\theta(x) =& u^{2}(x) - 1 - v^{2}(x) \nonumber\\
\theta'(x) =& 2(u(x)u'(x) - v(x)v'(x)) \nonumber\\
\theta''(x) =& 2(u'(x)^{2} - v'(x)^{2}) + 2(u(x)u''(x) - v(x)v''(x)).   
\end{align}
Now
\begin{align*}
\left(u'(x)^{2} - v'(x)^{2}\right)^{\prime} =& 2u'(x)u''(x) - 2v'(x)v''(x)\\
=& j_{x} \frac{2u'(x)u(x) - 2v'(x)v(x)}{\sqrt{u(x)^{2} - 1 - v(x)^{2}}}\\
 =& 2j_{x}(\sqrt{u(x)^{2} - 1 - v(x)^{2}})^{'} \\
 =& 2j_{x}(\sqrt{\theta(x)})^{'}.
\end{align*}
Integration from $0$ to $x$ gives
$$
u'(x)^{2} - v'(x)^{2} - (\underbrace{u'(0)^{2} - v'(0)^{2})}_{=\beta^{2} = -2j_{x}\gamma} = 2j_{x}\sqrt{\theta(x)} - 2j_{x}\underbrace{\sqrt{\theta(0)}}_{= 0},
$$
hence
\begin{equation}\label{prop3.1eq4}
u'(x)^{2} - v'(x)^{2} = 2j_{x}(\sqrt{\theta(x)} - \gamma).
\end{equation}
Furthermore we have
\begin{align}\label{prop3.1eq5}
u(x)\cdot u''(x) - v(x)\cdot v''(x) =& j_{x}\frac{u(x)^{2} - v(x)^{2}}{\sqrt{\theta(x)}} \nonumber\\
 =& j_{x}\frac{\theta(x) + 1}{\sqrt{\theta(x)}} \nonumber\\
 =&j_{x}\left(\sqrt{\theta(x)} + \frac{1}{\sqrt{\theta(x)}}\right ). 
\end{align}
We insert \eqref{prop3.1eq4} and \eqref{prop3.1eq5} into \eqref{prop3.1eq3},  and have
\begin{align*}
\theta''(x) =& 4j_{x}\left(\sqrt{\theta(x)} - \gamma\right) + 2j_{x}\left ( (\sqrt{\theta(x)} + \frac{1}{\sqrt{\theta(x)}}\right)\\
=& j_{x} \left(6\sqrt{\theta(x)} + \frac{2}{\sqrt{\theta(x)}} - 4\gamma\right)
\end{align*}
which is iii).
\end{proof}
We will need the following Lemma.

\begin{lem}\label{lem1} For $\gamma \in R$ define
$$
W_{1}(s) =: 1 - \gamma\sqrt{s} + s,\;\;\;\; s>0
$$
and let $[0,b)$ be the maximal interval on which
$$
I(z) =: \int\limits_{0}^{z}\frac{ds}{s^{1/4}\sqrt{W_{1}(s)}}
$$
exists. Then $I$ is strictly increasing, 
$$
I(b) =: \lim_{z\rightarrow b} I(z) \le \infty
$$
 exists with $I^{-1}: [0, I(b)) \rightarrow [0,b)$. We have
$$
I'(z) = \frac{1}{z^{1/4}\cdot \sqrt{W_{1}(z)}}, \;\;\;\;  I''(z) = -\frac{3}{4}\frac{W_{2}(z)}{z^{5/4}\sqrt[3]{W_{1}(z)}},
$$
where
$$
W_{2}(s) = \frac{1}{3} -\frac{2}{3}\gamma\sqrt{s} + s.
$$
The zeros of $W_{1}$ are:
$$
s_{11} =\left [\frac{\gamma}{2}\left (1 - \sqrt{1 - \frac{4}{\gamma^{2}}}\right )\right ]^{2}, \;\;\;\;   s_{12} = \left [\frac{\gamma}{2}\left (1 +\sqrt{1 - \frac{4}{\gamma^{2}}}\right )\right ]^{2}.
$$
$W_{1}$ has no real zeros for $\gamma < 2$, one zero $s_{11} = s_{12} = 1$ for $\gamma = 2$, two different real zeros $0<s_{11}<1<s_{12}$ for $\gamma>2$.

The  zeros of $W_{2}$ are: 
$$
s_{21} = \left [\frac{\gamma}{3}\left(1 - \sqrt{1 - \frac{3}{\gamma^{2}}}\right )\right ]^{2}, \;\;\;\;  s_{21} = \left [\frac{\gamma}{3}\left (1 + \sqrt{1 - \frac{3}{\gamma^{2}}}\right )\right ]^{2}.
$$
Then $W_{2}$ has no real zeros for $\gamma<\sqrt{3}$, 
one zero $s_{11} = s_{12} = \frac{1}{3}$   for $\gamma = \sqrt{3}$,
two different real zeros for $\gamma > \sqrt{3}$ with $0<s_{21}<s_{22}<1$ for $\sqrt{3}<\gamma<2$.
For $\gamma = 2$ we have $s_{22} = 1$, for $\gamma > 2$, $s_{21}<s_{11}<1<s_{22}$.\\

$I$ and $I^{-1}$ have the following properties according to the following 3 cases.
\begin{description}
\item[Case 1:] $\gamma < 2$. We have $b = \infty$, $I(\infty) = \infty$,
$$
I \in C[0,\infty)\cap C^{2}(0,\infty), \;\;\;\; I'(0) = \infty,
$$
$$
I^{-1}\in C^{1}[0,\infty)\cap C^{2}(0,\infty), \;\;\;\;\; (I^{-1})'(0) = 0.
$$
For $\gamma\le \sqrt{3}$, $I$ is strictly concave and $I^{-1}$ is strictly convex on $[0,\infty)$;

for $\sqrt{3}<\gamma<2$, $I$ is strictly concave on $[0,s_{21}]\cup [s_{22},\infty)$ and strictly convex on 
$[s_{21},s_{22}]$, and $I^{-1}$ is strictly convex on $[0, I(s_{21})]\cup [I(s_{22},\infty)]$ and strictly concave on 
$[I(s_{21}), I(s_{22})]$.

\item[Case 2:] $\gamma = 2$.  We have $W_{1}(s) = (1 - \sqrt{s})^{2}$, hence $b = 1$ and $I(1) = \infty$,
$$
I \in C[0,\infty)\cap C^{2}(0,\infty), \;\;\;\; I'(0) = \infty,
$$
$$
I^{-1}\in C^{1}[0,\infty)\cap C_{2}(0,\infty), \;\;\;\;  (I^{-1})'(0) = 0, \;\;\;\; I^{-1}(\infty) = 1.
$$
$I$ is strictly concave in $[0,s_{21}]$ and strictly convex on $[s_{21},1)$ and  $I^{-1}$ is strictly convex on $[0, I(s_{21})]$ 
and strictly concave on $[I(s_{21}), \infty)$.

\item[Case 3:] $\gamma> 2$. We have $W_{1}(s) = (\sqrt{s_{11}} - \sqrt{s})(\sqrt{s_{12}} - \sqrt{s})$ and hence $b = s_{11}$, $I(s_{11})<\infty$,
$$
I\in C[0,s_{11}]\cap C^{2}(0,s_{11}), \;\;\;\;I'(0) = I'(s_{11}) = \infty,
$$
$$
I^{-1}\in C^{1}[0,I(s_{11})]\cap C^{2}(0,I(s_{11})), \;\;\;\; (I^{-1})'(0) = (I^{-1})'(I(s_{11})) = 0.
$$
$I$ is strictly concave on $[0, s_{21}]$ and strictly convex on $[s_{21}, s_{11}]$, and $I^{-1}$ is strictly convex on $[0, I(s_{21})]$ 
and strictly concave on $[I(s_{21}), I(s_{11})]$.
\end{description}
\end{lem}
\begin{proof} 
The proof is straightforward.
\end{proof}
The picture is qualitatively the following:

\begin{center}
\begin{tikzpicture}
\draw[->] (0, 0) -- (7.5, 0);
\draw[->] (0, -0.5) -- (0, 5
);
\draw (0.1, 0) .. controls (1, 1.4) and (4.5, 1.6) .. (7.5, 1.9);
\node (a) at (1.63, 1.39) {$\bullet$};
\node (a) at (1.52, 1.8) {$\bullet$};
\node (a) at (1.2, 2.2) {$\bullet$};
\node (a) at (1.9, 3.9) {$\odot$};
\node (a) at (1.2, 3.3) {$\odot$};
\node (a) at (2.52, 4.42) {$\odot$};
\node (a) at (3.24, 5.62) {$\odot$};
\node (a) at (3.4, 3.18) {$\bullet$};
\node (a) at (1, 2.4) {$\bullet$};
\node (a) at (0.72, 2.6) {$\bullet$};
\node (a) at (0.42, 2.55) {$\bullet$};
\draw[dashed] (0.7, 0) -- (0.7, 2.6);
\draw[dashed] (1.9, 0) -- (1.9, 3.9) -- (0, 3.9);
\draw[dashed] (3.4, 0) -- (3.4, 3.1);
\draw[dashed] (3.8, 0) -- (3.8, 5.4);
\draw (0.07, 0) .. controls (0.5, 0.8) and (1, 1.1) .. (1.6, 1.4);
\draw (1.65, 1.42) .. controls (2.5, 1.8) and (4.5, 2.3) .. (7.5, 2.5);
\draw (0.04, 0) .. controls (0.5, 1.1) and (1, 1.5) .. (1.43, 1.78);
\draw (1.59, 1.87) .. controls (2, 2.2) and (4.5, 2.8) .. (7.5, 3.2);
\draw (0.02, 0) .. controls (0.2, 1.1) and (0.7, 1.8) .. (1.2, 2.22);
\draw (1.22, 2.22) .. controls (2, 2.5) and (2.7, 2.6) .. (3.4, 3.2);
\draw (3.4, 3.2) .. controls (4, 4) and (6, 5) .. (6.9, 5.2);
\draw (0.01, 0) .. controls (0.1, 1.4) and (0.8, 2.3) .. (1, 2.4);
\draw (1.06, 2.44) .. controls (2, 2.7) and (3.6, 4) .. (3.8, 5.7);
\draw (0, 0) .. controls (0.1, 1.4) and (0.5, 2.3) .. (0.7, 2.6);
\draw (0.7, 2.62) .. controls (1, 2.7) and (1.8, 3.4) .. (1.9, 3.8);
\draw[dashed] (2, 3.95) .. controls (2.2, 4) and (3.2, 5) .. (3.4, 6.2);
\draw (0, 0) .. controls (0.1, 1.6) and (0.3, 2.3) .. (0.4, 2.6);
\draw (0.4, 2.6) .. controls (0.8, 2.8) and (1, 3) .. (1.2, 3.2);
\draw[dashed] (1.28, 3.25) .. controls (1.4, 3.3) and (1.6, 3.6) .. (1.8, 3.8);
\node (a) at (-0.6, 3.9) {$I(S_{11})$};
\node (a) at (0, 5.2) {$I$};
\node (a) at (1.3, 4.4) {$\underset{\overbrace{\hspace{1.2cm}}}{r > 2}$};
\node (a) at (4.3, 5.8) {$\gamma = 2$};
\node (a) at (8.2, 2.6) {$\gamma = \sqrt{3}$};
\node (a) at (8.2, 2) {$\gamma < \sqrt{3}$};
\node (a) at (8.4, 4.2) {$\left.\begin{matrix}
	\\
	\\
	\\
	\\
	\end{matrix} \right\} \gamma \in (\sqrt{3},2)$};
\node (a) at (0.7, -0.3) {$S_{21}$};
\node (a) at (1.9, -0.3) {$S_{11}$};
\node (a) at (3.4, -0.3) {$S_{22}$};
\node (a) at (3.8, -0.3) {$1$};
\node (a) at (7.8, 0) {$S$};
\end{tikzpicture}

Turning points: $\odot$ Points with $I'=\infty$  $\;\bullet$ Points with $I''=0$.
\end{center}
Let
$$
I_{\gamma}(z) =: \int\limits_{0}^{z} \frac{ds}{s^{1/4}\sqrt{1 - \gamma\sqrt{s} + s}}.
$$
Then
$$
\frac{\partial}{\partial\gamma}I_{\gamma}(z) = \int\limits_{0}^{z}\frac{-\frac{1}{2}(-\sqrt{s})}{s^{1/4}\sqrt[3]{1 - \gamma\sqrt{s} + s}}ds > 0.
$$
$$
s_{11,2}(\gamma) =\left [\frac{\gamma}{2} \pm \sqrt{\frac{\gamma^{2}}{4} - 1}\right ]^{2}, \;\;\;\;    s_{21,2}(\gamma) = \left [\frac{\gamma}{3} \pm \sqrt{\frac{\gamma^{2}}{9} - \frac{1}{3}}\right ]^{2}.
$$
\begin{center}
\begin{tikzpicture}
\draw (5.7, 0.2) .. controls (2.3, 0.7) and (2.5, 1.4) .. (4.2, 4);
\draw[dashed] (5.85, 0.1) .. controls (0.9, -0.1) and (2.22, 0.8) .. (3, 1.5);
\draw[dashed] (3, 1.5) .. controls (3.5, 2) and (4.4, 2) .. (5, 3.4);
\draw[->] (0, 0) -- (6, 0);
\draw[->] (0, 0) -- (0, 4);
\draw (0, 0.5) -- (5.7, 0.5);
\draw (0, 1.5) -- (5.7, 1.5);
\draw (1.5, 0) -- (1.5, 3);
\draw (2.25, 0) -- (2.25, 0.5);
\draw (3, 0) -- (3, 3);
\draw (4.5, 0) -- (4.5, 3);
\node (a) at (1.5, -0.3) {$1$};
\node (a) at (3, -0.3) {$2$};
\node (a) at (-0.4, 0.5) {$1/3$};
\node (a) at (-0.3, 1.5) {$1$};
\node (a) at (2.25, -0.3) {$\sqrt{3}$};
\node (a) at (4.5, -0.3) {$3$};
\node (a) at (6.3, -0.1) {$\gamma$};
\node (a) at (6.3, 0.25) {$S_{21} (\gamma)$};
\node (a) at (5.2, 0.7) {$S_{11} (\gamma)$};
\node (a) at (5.5, 3.4) {$S_{22} (\gamma)$};
\node (a) at (4.8, 4) {$S_{12} (\gamma)$};
\end{tikzpicture}

Asymptotics: $S_{12}(\gamma)\sim \gamma^2$, $S_{22}(\gamma)\sim\frac{4}{9}\gamma^2$, for large $\gamma$.
\end{center}

\begin{rem}\label{rem1}  Let $f: [0,\varepsilon)\rightarrow [0, f(\varepsilon))$ be strictly increasing with $f(\varepsilon) = lim f(x)$, 
so that $f^{-1}: [0, f(\varepsilon))\rightarrow [0,\varepsilon)$ exists.Let $c>0$. Then $\frac{1}{c}f: [0,\varepsilon)\rightarrow [0,\frac{1}{c}f(\varepsilon))$ is
invertible with $(\frac{1}{c}f)^{-1} = f^{-1}(c): [0,\frac{1}{c}f(\varepsilon))\rightarrow [0,\varepsilon)$. If $f^{-1}$ is convex (concave) on some 
subinterval $[\alpha, \beta]$, then $(\frac{1}{c}f)^{-1}$ is convex (concave) on $[\frac{\alpha}{c}, \frac{\beta}{c}]$.
\end{rem}
\begin{prop}\label{prop2} For every $\gamma\in R$ there exists a unique solution on $[0,\infty)$ of the initial value problem
\begin{align}\label{prop3.3eq1}
D'' =& j_{x}\left(6\sqrt{D} + \frac{2}{\sqrt{D}} - 4\gamma\right) = \frac{6 j_{x}}{\sqrt{\theta}} W_{2}(D), \nonumber\\ 
D(0) =& D'(0) = 0, \quad       W_{2}(s) = \frac{1}{3} - \frac{2}{3}\gamma\sqrt{s} + s.
\end{align}
In particular, if $b$ is the quantity defined in Lemma \ref{lem1}, then 
\begin{equation}\label{prop3.3eq2}
D =: \left(\frac{1}{8j_{x}}I\right)^{-1} = I^{-1}(\sqrt{8j_{x}}): \left[0, \frac{1}{\sqrt{8j_{x}}} I(b)\right) \longrightarrow [0,b)
\end{equation}
is strictly increasing and has the properties i), ii), iii) of Proposition \ref{prop1} for $a =: \frac{1}{\sqrt{8j_{x}}} I(b)$, and 
is a solution of 
$$
D'^{2} = 8 j_{x}(\sqrt[3]{D} + \sqrt{D} - \gamma D) = 8 j_{x} \sqrt{D}W_{1}(D),
$$
$$
D(0) = 0, \;\;\;\;\; W_{1}(s) = 1 - \gamma\sqrt{s} + s.
$$
Moreover, for varying $\gamma, D$ has the following properties:
\begin{description}
\item[Case 1:] $\gamma< 2$.   $b = \infty$, $D\in C^{1}[0,\infty)\cap C^{2}(0,\infty)$, $D(\infty) = \infty$, $D'(0) = 0$.

For $\gamma\le \sqrt{3}$, $D$ is strictly convex on $[0,\infty)$; for $\sqrt{3}<\gamma <2$, $D$ is strictly convex on 
$[0,\frac{1}{\sqrt{8j_{x}}} I(s_{21})] \cup [\frac{1}{\sqrt{8j_{x}}} I(s_{22}), \infty)$ and strictly concave on 
$[\frac{1}{\sqrt{8j_{x}}} I(s_{21}), \frac{1}{\sqrt{8j_{x}}} I(s_{22})]$.

\item[Case 2:] $\gamma=2$, $b = 1$, $D \in C^{1}[0,\infty)\cap C^{2}(0,\infty)$, $D'(0) = 0$, $D(\infty) = 1$, $D$ is strictly convex on $[0, \frac{1}{\sqrt{8j_{x}}} I(s_{21})]$
and strictly concave on $[\frac{1}{\sqrt{8j_{x}}} I(s_{21}), \infty]$. 

\item[Case 3:] $\gamma >2$. $b = s_{11}$, $a =: \frac{1}{\sqrt{8j_{x}}} I(b) < \infty$, $D\in C^{1}[0, a]\cap C^{2}(0,a]$, $D'(0) = D'(a) = 0$, $D$ is strictly convex on 
$[0, \frac{1}{\sqrt{8j_{x}}} I(s_{21})]$ and strictly concave on $[\frac{1}{\sqrt{8j_{x}}} I(s_{21}), a]$. $D$ can be extended to a solution of \eqref{prop3.3eq1} on $(0,\infty)\backslash 2a \mathbb{N}$ to 
a function $D\in C^{1}[0,\infty)\cap C^{2}((0,\infty)\backslash 2 a \mathbb{N})$. First by extending $D$ from $[0,a]$ to $[a,2a]$ by the definition 
\begin{equation}\label{prop3.3eq3}
D(x) =: D(2a - x), \;\;\;\;  x\in [a,2a] 
\end{equation}
and then by extending $D$ from $[0,2a]$ to $[2a,\infty)$ periodically.
\end{description}
\end{prop}
\begin{proof} Let us assume that $D$ has properties i) - iii) of Proposition \ref{prop1} on some $[0, \tilde{a})$. If we multiply \eqref{prop3.1eq2} by $2D'(x)$, we obtain
\begin{align*}
(D'(x)^{2})' =& 2 D'(x)D''(x)\\
=&j_{x} \left (12 \sqrt{D(x)} D'(x) + \frac{4}{\sqrt{D(x)}} D'(x) - 8\gamma D'(x)\right )\\		
=& j_{x} (8\sqrt[3]{D(x)} + 8\sqrt{D(x)} - 8\gamma D(x))'.
\end{align*}
Integration from $0$ to $x$ gives, with $D(0) = D'(0) = 0$,
\begin{align*}
D'(x)^{2} =& 8j_{x} (\sqrt[3]{D(x)} + \sqrt{D(x)} - D(x))\\
=& 8j_{x}\sqrt{D(x)} W_{1}(D(x)).
\end{align*}
For small $\tilde{a}>0$, $W_{1}(D(x))>0$ and
$$
\frac{1}{\sqrt{8j_{x}}} \frac{D'(x)}{D(x)^{1/4}\sqrt{W_{1}(D(x))}} = 1.
$$
Integration from $0$ to $x$ gives
$$
\frac{1}{\sqrt{8j_{x}}}\int\limits_{0}^{D(x)} \frac{ds}{s^{1/4} \sqrt{W_{1}(s)}} = x,
$$
that is, $\frac{1}{\sqrt{8j_{x}}} I(D(x)) = x$ or \eqref{prop3.3eq2}. Hence, if we define $D$ by \eqref{prop3.3eq2}, $D$ is the desired function with the properties i) - iii) of 
Proposition \ref{prop1}. The Cases 1-3 follow from Lemma \ref{lem1} in condition with the Remark \ref{rem1}. The extension of $D$ on $[0,a]$ to the interval 
$[a, 2a]$ by the definition \eqref{prop3.3eq3} yields a solution of \eqref{prop3.3eq1} on $[a,2a]$ because for $x\in [a, 2a]$
\begin{align*}
D''(x) = D''(2a - x) =& j_{x}\left(6\sqrt{\theta(2a - x)} + \frac{2}{\sqrt{\theta(2a - x)}} - 4\gamma \right)\\
=& j_{x}\left(6\theta(x) + \frac{2}{\sqrt{\theta(x)}} - 4\gamma\right).
\end{align*}
\end{proof}
\begin{lem}\label{lem2} Let $D$ be the solution of Proposition \ref{prop2} then we get asymptotic
$$
D(x) = \left(\frac{3}{\sqrt{2}}\right)^{4/3} j_{x}^{2/3}x^{4/3} (1 + O(x^{2/3})).  \qquad    x\rightarrow 0,
$$
In particular, $\frac{1}{\sqrt{D}}$ is integrable at $0$.
\end{lem}
\begin{proof} Using 
$$
(1 - y)^{-1/2} = 1 + \frac{1}{2}y + \frac{3}{8}y^{2} + O(y^{3})\;\;\;\;{\rm for}\;\;\;\; y\rightarrow 0
$$
together with the definition of $I$, we have for $z\rightarrow 0$
\begin{align*}
I(z) =& \int\limits_{0}^{z} \frac{ds}{s^{1/4}\sqrt{1 - \gamma s^{1/2} + s}} = \int\limits_{0}^{z} s^{-1/4} \frac{ds}{\sqrt{1 - s^{1/2}(\gamma - s^{1/2})}}\\
=&\int\limits_{0}^{z}s^{-1/4}\left [1+\frac{1}{2}s^{1/2}(\gamma - s^{1/2}) + \frac{3}{8}s(\gamma^{2} - 2\gamma s^{1/2} + s) + O(s^{3/2})\right ]ds\\
= &\int\limits_{0}^{z}\left [s^{-1/4} + \frac{1}{2}\gamma s^{1/4} + (-\frac{1}{2} + \frac{3}{8} \gamma^{2})s^{3/4} + O(s^{5/4})\right ]ds\\
=& \frac{4}{3}z^{3/4} + \frac{2}{5}z^{5/4} + \left(-\frac{2}{7} + \frac{3}{14}\gamma^{2}\right)z^{7/4}  + O(z^{9/4})\\
=&\frac{4}{3}z^{3/4}\left(1 + \frac{3}{10}z^{1/2} + O(z)\right) ,\\
\frac{1}{\sqrt{8j_{x}}} I(z) =& \frac{1}{\sqrt{8j_{x}}}\cdot \frac{4}{3} z^{3/4}\left (1 + \frac{3}{10} z^{1/2} + O(z)\right ),\\
x =& \frac{1}{\sqrt{8j_{x}}} \frac{4}{3} D(x)^{3/4}\left(1 + \frac{3}{10}D(x)^{1/2} + O(D(x)\right),\\
D(x)^{3/4} =& \frac{3}{\sqrt{2}} \sqrt{j_{x}} x (1 + O(D(x)^{1/2}),  \;\;\;\;\;  D(x) = O(x^{4/3})\\
D(x)^{3/4} =& \frac{3}{\sqrt{2}} \sqrt{j_{x}} x (1 + O(x^{2/3})),\\
D(x) =& \left(\frac{3}{\sqrt{2}}\sqrt{j_{x}}\right)^{4/3}x^{4/3}\left(1+O(x^{2/3})\right)^{4/3}\\ =& \left(\frac{3}{\sqrt{2}} \sqrt{j_{x}}\right)^{4/3} x^{4/3} (1 + O(x^{2/3})).
\end{align*}
\end{proof}

The last asymptotics raises an interesting question: Is the critical value of
the current $j_{x}$ in the case $a = 0$ (zero magnetic field) is compared to the Child-Langmuir law which relates the current to the electrostatic potential $\varphi$ at the anode by the 3/2 power law?

The fundamental Child-Langmuir limit on the maximum current density
in a vacuum between two in infinite parallel electrodes is one of the most well
known and often applied rules of plasma physics. Child and Langmuir \cite{Langmuir1931}
first derived classical space-charge-limited emission for two such electrodes
in a vacuum separated by a distance $D$ and a potential difference $V$. We consider a second-order nonlinear differential equation for the potential:
\begin{equation}
    \varphi^{''} = = \frac{j_{CL}}{\epsilon_{0}}\sqrt{\frac{m}{2e}}\sqrt{\frac{1}{\varphi(z)}},
\end{equation}
where $m$ and $e$ are the electron mass and charge, respectively, $\varphi(z)$ is the potential field in the gap. A first integration of Eq. (27) can be performed after first multiplying both sides by $d\varphi/dz$:
\begin{equation}
\left (\frac{d\varphi}{dz}\right)^2= \frac{4j_{CL}}{\epsilon_{0}}\sqrt{\frac{2e}{m}}\sqrt{\varphi(z)} + K= \frac{4j_{CL}}{\epsilon_{0}}\sqrt{\frac{2e}{m}}\sqrt{\varphi(z)} + \left (\frac{d\varphi}{dz}|_{z=0}\right )^2
\end{equation}
where the constant of integration $K$ is found by applying the boundary condition
of the potential at the cathode, $\varphi(0)=0$. We integrate Eq. (28) to find the potential in this space-charge-limited gap:
\begin{equation}
\varphi(z)= \left (\frac{3}{2}\right )^{4/3}\left (\frac{j_{CL}}{\epsilon_{0}}\right )^{2/3}\left(\frac{m}{2e}\right )^{1/3}z^{4/3} + K.
 \end{equation}
If we apply the boundary condition of the fixed potential anode, $\varphi(D)=V$, we can solve Eq. (32) for the space-charge-limited current density:
\begin{equation}
    j_{CL}=\frac{4}{9}\epsilon_{0}\sqrt{\frac{2e}{m}}\frac{V^{3/2}}{D^{2}}.
\end{equation}
Eq. (30) is the classical Child-Langmuir 3/2 power law in electrostatic case,
when magnetic field is zero.
\begin{prop}\label{prop3}{\rm [15]} Let $0< c\leq j_{x}\leq j_{x}^{max}$, $a=0$. Then equation 
$$
\varphi^{''}= j_{x}\frac{1+\varphi}{\sqrt{\varphi(2+\varphi)}}, \;\; \varphi(0)=0, \; \varphi(1) = \varphi_{L}
$$
has a lower positive solution 
\begin{equation}
    u_{0}=\delta^{2}x^{4/3}
\end{equation}
if 
\begin{equation}
    4\delta^{3}\geq 9 j_{x}^{max}(1+\delta^{2}) / \sqrt{2+ \delta^{2}}
\end{equation}
and an upper positive solution 
\begin{equation}
    u^{0}= \alpha + \beta x \;\;\; (\alpha, \beta >0)
\end{equation}
with 
\begin{equation}
\varphi_{L} \geq \delta^{2}
\end{equation}
where $\delta$ is defined from (32).
\end{prop}

Consider the limit case of inequality (32)
$$
 \delta^{3} = \frac{9}{4} j_{x}^{max}(1+\delta^{2}) / \sqrt{2+ \delta^{2}}
$$
or
\begin{equation}
    \delta = \left (\frac{9}{4}j_{x}^{max} K(\delta)\right )^{1/3}
\end{equation}
$K(\delta)=(1+\delta^{2})/\sqrt{2+\delta^{2}}$. Substituting (35) in (31) we obtain
$$
u_{0}= \left (\frac{9}{4}j_{x}^{max} K(\delta)\right )^{2/3} x^{4/3}
$$
From the last equation, applying condition $u_{0}(D) = V$ we obtain
\begin{equation}
    j_{x}^{max}= j_{CL}= \frac{4}{9}\frac{1}{K(\delta)}\frac{V^{3/2}}{D^{2}}.
\end{equation}
Eq. (36) is the Child-Langmuir law in dimensionless variables, which corresponds equation (30) up to a constant $K(\delta)$ with conditions (32), (34). 

\begin{prop}\label{prop4} Let $D$ be the solution of the initial value problem
$$
D'' = j_{x}\left(6\sqrt{D} + \frac{2}{\sqrt{D}} - 4\gamma\right), \;\;\;\;\;D(0) = D'(0) = 0
$$
according to Proposition \ref{prop2} on $[0,\varepsilon)$.  Assume $(U,V)$ is a $C^{1}[0,\varepsilon) \cap C^{2}(0,\varepsilon)$ - solution of
\begin{align*}
u'' =& j_{x} \frac{U}{\sqrt{D}} , \qquad U(0) = 1, \qquad U'(0) = 0\\
V'' =& j_{x}\frac{V}{\sqrt{D}}, \qquad V(0) = 0, \qquad V'(0) = \beta, \qquad \beta^{2} = 2j_{x}\gamma
\end{align*}
on $[0,\varepsilon)$. Then $U^{2} - 1 - V^{2} = D$.
\end{prop}
\begin{proof} We let $\theta =: U^{2} - 1 - V^{2}$. Then
\begin{align}\label{prop3.5eq1}
\theta^{'} =& 2(UU^{'} - VV^{'}) \nonumber\\
\theta'' =& 2(U'^{2} - V'^{2}) + 2(UU'' - VV'').           
\end{align}
Now
$$
(U'^{2} - V'^{2})' = 2U'U'' - 2V'V'' = j_{x} \frac{2U'U - 2V'V}{\sqrt{D}} = j_{x} \frac{\theta'}{\sqrt{D}}.
$$
Integration from $0$ to $x$ gives
\begin{equation}\label{prop3.5eq2}
U'^{2}(x) - V'^{2}(x) = j_{x}\int\limits_{0}^{x} \frac{\theta'(t)}{\sqrt{D}} dt - 2\gamma j_{x}.
\end{equation}
Furthermore, we have
\begin{equation}\label{prop3.5eq3}
UU'' - VV'' = j_{x} \frac{U^{2} - V^{2}}{\sqrt{D}} = j_{x} \frac{\theta + 1}{\sqrt{D}}.          
\end{equation}
We insert \eqref{prop3.5eq2} and \eqref{prop3.5eq3} in \eqref{prop3.5eq1} and obtain
\begin{equation}\label{prop3.5eq4}
\theta'' = 2j_{x}\left(\int\limits_{0}^{x} \frac{\theta'(t)}{\sqrt{D}}dt -2\gamma +
\frac{\theta + 1}{\sqrt{D}}\right)   , \;\;\;\;\theta(0) = \theta'(0) = 0.
\end{equation}
By assumption $D$ satisfies, with Lemma \ref{lem2}
\begin{align}\label{prop3.5eq5}
D'' =& j_{x}\left(6\sqrt{D} + \frac{2}{\sqrt{D}} - 4\gamma\right) \nonumber\\
=& j_{x}\left(4\sqrt{D} - 4\gamma + 2\frac{D +1}{\sqrt{D}}\right) \nonumber\\
=& 2 j_{x} \left(\int\limits_{0}^{x}\frac{D'(t)}{\sqrt{D}}dt - 2\gamma + \frac{D + 1}{\sqrt{D}}\right), \qquad D(0) = D'(0) =0.   
\end{align}
Equations \eqref{prop3.5eq4} and \eqref{prop3.5eq5} imply: $\theta$ and $D$ are both solutions of the linear problem
$$
y'' = 2j_{x}\left( \int\limits_{0}^{x}\frac{y'(t)}{\sqrt{D(t)}} dt - 2\gamma + \frac{y + 1}{\sqrt{D}}\right), \;\;\;\;  y(0) = y'(0) = 0.
$$
We need to show: $z =: \theta - D  = 0$. We have
$$
z''(x) = 2 j_{x}\left( \int\limits_{0}^{x} \frac{z'(t)}{\sqrt{D}} dt + \frac{z(x)}{\sqrt{D}}\right), \;\;\;\;\; z(0) = z'(0) = 0.
$$
Integration from $0$ to $x$ yields with partial integration and $\mathcal{D}(x) =: \int\limits_{0}^{x}\frac{1}{\sqrt{D(s)}} ds$,
\begin{align*}
z'(x) =& 2 j_{x}\left(\int\limits_{0}^{x} 1\cdot \int\limits_{0}^{s}\frac{z'(t)}{\sqrt{D(t)}} dt + \int\limits_{0}^{x}\frac{z(s)}{\sqrt{D(s)}} ds \right)\\
=&2 j_{x}\left( x\cdot \int\limits_{0}^{x} \frac{z'(t)}{\sqrt{D(t)}} dt - \int\limits_{0}^{x} s\frac{z'(s)}{\sqrt{D(s)}} ds + \int\limits_{0}^{x}\frac{1}{\sqrt{D(s)}}\cdot \int\limits_{0}^{s} z'(t) dt \right)\\
=& 2 j_{x} \left( \int\limits_{0}^{x} (x - s) \frac{1}{\sqrt{D(s)}} z'(s) ds + \mathcal{D}(x)\cdot \int\limits_{0}^{s} z'(t) dt - \int\limits_{0}^{x} \mathcal{D} (s) z'(s) ds\right)\\
=& 2 j_{x} \int\limits_{0}^{x} \left(\frac{x - s}{\sqrt{D(s)}} + \mathcal{D}(x) - \mathcal{D} (s)\right) z'(s) ds.
\end{align*}
For every $\eta\in (0,\varepsilon)$ exists $g_{\eta} \in L^{1} [0, \varepsilon - \eta]$ such that for all $x\in [0, \varepsilon - \eta]$ and $s \in [0,x] $
$$
2 j_{x} \left | \frac{x - s}{\sqrt{D(s)}} + \mathcal{D}(x) - \mathcal{D}(s)\right | \le g_{\eta}(s). 
$$
Hence for all $x\in [0,\varepsilon - \eta]$
$$
|z'(x)| \le \int\limits_{0}^{x} g_{\eta}(s)|z'(s)| ds.
$$
Let $\rho(x) =: \int\limits_{0}^{x} g_{\eta}(s) |z'(s)| ds$. Then
$$
\rho'(x) = g_{\eta}(x) |z'(x)| \le g_{\eta}(x) \rho(x), \;\;\;\;\;   x\in (0,\varepsilon - \eta].
$$
For all $\delta, x \in (0,\varepsilon - \eta)$ it follows
$$
\rho(x) \leq \underbrace{\rho(\delta)}_{\rightarrow 0} \cdot \underbrace{e^{\int\limits_{\delta}^{x} g_{\eta}(s)ds}}_{{\rm bounded\;\; as}\;\; \delta\rightarrow 0} \rightarrow 0.   \;\;\;  (\delta\rightarrow 0).
$$
Hence $\rho(x) = 0$ on $(0,\epsilon - \eta)$ and $|z'(x)|\leq \rho(x)= 0$ on 
$(0,\epsilon - \eta)$. Because $z(0)= 0$, we have $z(x) = 0$ on $[0, \epsilon - \eta)$ for all $\eta\in (0,\epsilon)$ which implies $z = 0$ on $[0,\epsilon)$. Hence
$\theta = D$. 
\end{proof}
We now prove that the assumption of Proposition \ref{prop3} (the existence of $(U, V)$) is satisfied.

\begin{thm}  a) Let $D\in C(0,\epsilon)$ be positive and $\frac{1}{\sqrt{D}}$
integrable at $0$. Then:

For $\alpha, \beta \in R$ there exists a unique solution $(u,v)\in(C^{1} [0,\epsilon)\cap    C^{2}(0,\epsilon))^{2}$ of the initial value problem
\begin{equation}\label{thm3.6eq1}
\begin{array}{l}
u'' = j_{x}\frac{u}{\sqrt{D}}, \qquad  u(0) = 1, \qquad  u'(0) =\alpha\\
v'' =j_{x}\frac{v}{\sqrt{D}},  \qquad  v(0) = 0, \qquad v'(0) =\beta. 
\end{array}
\end{equation}
b) If $\alpha =0$, $\gamma$ is defined by $\beta^{2} = 2j_{x}\gamma$ and $D$ the solution of the initial value problem
\begin{equation}\label{thm3.6eq2}
D'' = j_{x} \left (6\sqrt{D} + \frac{2}{\sqrt{D}} - 4\gamma \right ), \qquad D(0) = D'(0) = 0  
\end{equation}
according to Proposition \ref{prop2}, then $(u,v)$ obtained in a) is the solution of 
\begin{equation}\label{thm3.6eq3}
\begin{array}{l}
u'' = j_{x}\displaystyle\frac{u}{\sqrt{u^{2} - 1 - v^{2}}},      \qquad u(0) = 1, \qquad u(0) = 0\\
v'' = j_{x} \displaystyle\frac{v}{\sqrt{u^{2} - 1 - v^{2}}},    \qquad v(0) = 0, \qquad v'(0) =\beta     
\end{array}
\end{equation}
\end{thm}
\begin{proof}
 a) Assume that $(u,v)\in (C^{1}[0,\epsilon) \cap C^{2}(0,\epsilon))^{2}$ is a solution of \eqref{thm3.6eq1} on $[0,\epsilon)$. Integrating \eqref{thm3.6eq1} twice from $0$ to $x$ and 
using initial conditions in \eqref{thm3.6eq1} we obtain
\begin{align}\label{thm3.6eq4}
u'(x) =& \alpha + j_{x}\int\limits_{0}^{x} \frac{u(t)}{\sqrt{D(t)}}dt\nonumber\\
u(x) =& 1+ \alpha x + j_{x}\int\limits_{0}^{x}\int\limits_{0}^{s}\frac{u(t)}{\sqrt{D(t)}}dtds\nonumber\\
& = 1 + \alpha x + j_{x}\int\limits_{0}^{x}(x - s) \frac{u(s)}{\sqrt{D(s)}} ds 
\end{align}
\begin{align}\label{thm3.6eq5}
  v'(x) =&\beta +j_{x}\int\limits_{0}^{x}\frac{v(t)}{\sqrt{D(t)}} dt\nonumber\\
v(x) =& \beta x + \int\limits_{0}^{x}\int\limits_{0}^{s} \frac{v(t)}{\sqrt{D(t)}}dt ds \nonumber\\
& = \beta x +j_{x}\int\limits_{0}^{x} (x-s) \frac {v(s)}{\sqrt{D(s)}} ds. 
\end{align}     
Consequently, every solution $(u, v)\in C[0,\epsilon)^{2}$ of  equations \eqref{thm3.6eq4} and \eqref{thm3.6eq5} is 
in $(C^{1}[0,\epsilon)\cap C^{2}(0,\epsilon))^{2}$ and satisfies \eqref{thm3.6eq1}. It is sufficient to prove the existence of such solutions on a small interval $[0,\delta]$, because there always exists a solution of the linear differential equations \eqref{thm3.6eq1} on 
$[\delta, \epsilon)$. To this end, we use Banach's fixed point theorem as follows.

We define $T_{1}: C[0,\delta] \rightarrow C[0,\delta]$ by
$$
T_{1} w(x) := 1 + \alpha x + j_{x}\int\limits_{0}^{x}(x - s) \frac{w(s)}{\sqrt{D(s)}}ds, \;\;\;\; w\in C[0,\delta].
$$
If $\parallel w\parallel \leq M$, then
$$
|T_{1}w(x)|\leq 1+|\alpha| x + j_{x}\cdot x \int\limits_{0}^{x}\frac{M}{\sqrt{D(s)}}\;ds
\leq 1+ |\alpha|\delta + j_{x}\delta M \frac{1}{\parallel \frac{1}{\sqrt{D}}\parallel_{1}}.
$$
For $w_{1}, w_{2} \in C[0,\delta]$ we have
$$
|(T_{1} w_{1} - T_{1}w_{2})(x)| \leq j_{x}|x| \int\limits_{0}^{x} \frac{\parallel w_{1} - w_{2}\parallel}{\sqrt{D(s)}}\; ds \leq j_{x}\delta \parallel w_{1} - w_{2}\parallel 
\cdot \frac{1}{\parallel \sqrt{\frac{1}{D}}\parallel_{1}}.
$$
Similarly, we define $T_{2}: C[0,\delta] \rightarrow C[0,\delta]$ by
$$
T_{2}w(x) := \beta x + j_{x}\int\limits_{0}^{x}(x - s) \frac{w(x)}{\sqrt{D(s)}} \;ds, \;\;\;\;
w\in C[0,\delta].
$$
If $\parallel w\parallel \leq M$, then
$$
|T_{2}w(x)| \leq |\beta| x + j_{x}\cdot x\int\limits_{0}^{x}\frac{M}{\sqrt{D(s)}}\;ds \leq |\beta| \delta + j_{x} \delta\cdot M\frac{1}{\parallel\frac{1}{\sqrt{D}}\parallel_{1}}.
$$
For $w_{1}, w_{2} \in C[0,\delta]$ we have
$$
|(T_{2}w_{1} - T_{2}w_{2})(x)|\leq j_{x}|x|\int\limits_{0}^{x}\frac{\parallel w_{1} - w_{2}\parallel}{\sqrt{D(s)}}\;ds \leq j_{x}\delta\parallel w_{1} - w_{2}\parallel \cdot \frac{1}{\parallel\frac{1}{\sqrt{D}}\parallel_{1}}.
$$  
Consequently, $T_{1}, T_{2}$ are contracting mappings from $\{w\in C[0,\delta], \parallel w\parallel \leq M\}$ into itself if 
$$
1 + \delta \left ( |\alpha| + j_{x}M\frac{1}{\parallel\frac{1}{\sqrt{D}}\parallel_{1}}\right ) \leq M
$$
$$
\delta \left( |\beta| + j_{x}M\frac{1}{\parallel\frac{1}{\sqrt{D}}\parallel_{1}}\right ) \leq M
$$ 
$$
\delta j_{x}M\frac{1}{\parallel\frac{1}{\sqrt{D}}\parallel_{1}} <1.
$$
For every $M>1$ there exists $\delta$ satisfying these conditions. This proves a).

b) Let $\alpha, \beta \in R$ and define $\gamma$ by $\alpha^{2} - \beta^{2} = -2\gamma j_{x}$. It follows from Proposition \ref{prop2}, that the solution $D$ of the initial value problem \eqref{thm3.6eq2} is in $C(0,\delta)$ and positive and from Lemma \ref{lem2} that $\frac{1}{\sqrt{D}}$ is integrable at $0$. It follows from a), that there exists a unique solution $(u,v) \in (C_{1}[0,\epsilon)\cap C^{2}(0,\epsilon))^{2}$ of the initial value problem \eqref{thm3.6eq1}. By the choice of $\gamma$, $(u, v)$ is a solution of the initial value problem \eqref{prop3.5eq1}, and Proposition \ref{prop3} then says that $u^{2} - 1 - v^{2} = D$, that is, $(u,v)$ solves \eqref{thm3.6eq3}.
\end{proof}
\section{The Isolated Case}
In this Section we study the magnetically insulated diode, $\theta<0$.We are looking for the conditions under which this phenomenon occurs. 
The electron flow depends on $j_{x}, \beta, k$,  the boundary conditions at the anode and the unknown point $x_{d}$.

We search a point $x_{d}$

\begin{equation}\label{ISC-1}
\begin{array}{ll}
\forall x\in (0,x_{d}],\quad & \theta'(x)\geq 0 \,\, \wedge \,\, \theta(x) \geq 0,\\
\forall x\in (x_{d},1], \quad & \theta'(x)<0 \,\, \wedge \,\, \theta(x) < 0.  
\end{array}    
\end{equation}

In order to implement this process, we consider the system \eqref{MID-2} in terms of $\varphi$ and $a$ and we write it via the effective potential. 

\begin{equation}\label{ISC-2}
    \theta' = 2\left((1+\varphi)\varphi' - aa' \right)          \hspace{3cm}
\end{equation}
\begin{equation}\label{ISC-3}
    \theta'' = 2\left((1+\varphi)\varphi'' - aa'' + (\varphi')^2-(a')^2 \right)
\end{equation}
\begin{equation}\label{ISC-4}
    \varphi'' = j_x (1+\varphi) \theta^{-\frac{1}{2}}
\end{equation}
\begin{equation}\label{ISC-5}
    a'' = j_x a \title{}\theta^{-\frac{1}{2}}.            \hspace{1cm}
\end{equation}

We obtain from  \eqref{ISC-2}, \eqref{ISC-3}, \eqref{ISC-4}, \eqref{ISC-5}:

$$\theta'' = 2\left( j_x \theta^{-\frac{1}{2}} (\theta + 1) + (\varphi')^2-(a')^2 \right),$$
$$j_x \theta^{\frac{-1}{2}}\theta'= 2\varphi ''\varphi'- 2a''a' $$
$$
\begin{array}{c}
\left( 2j_x\theta^{\frac{1}{2}} \right) '=\left[ (\varphi')^2-(a')^2 \right.]'  \\
  2j_x\theta^{\frac{1}{2}} + c_1 = (\varphi')^2-(a')^2 \\
\end{array}
$$
\begin{equation}\label{theta2d}
\dfrac{\theta''}{2}= j_x \theta^{\frac{-1}{2}} ( \theta + 1 )+ 2j_x\theta^{\frac{1}{2}} + c_1
\end{equation}
\begin{equation}\label{thetacon}
\begin{array}{c}
\theta(0)= (1+0)^2-1-0=0 \hspace{1cm}\\
\theta(1)= (1+\varphi_L)^2-1-a_L=\theta_L \hspace{0.3cm} \\ \theta'(0)=2(1+0)\beta - 2(0)(\alpha)=2\beta
\end{array}
\end{equation}

Integrating ODE \eqref{theta2d} with conditions \eqref{thetacon}  we get
$$ (\theta ')^2 = 8 j_x \theta^{\frac{3}{2}} + 8 j_x \theta^{\frac{1}{2}} + 4 c_1 \theta + 2c_2$$
\begin{equation}\label{theta1d}
    (\theta ')^2 = k \theta + 8 j_x \theta^{\frac{3}{2}} + 8 j_x \theta^{\frac{1}{2}}  + 4\beta^2.
\end{equation}

\subsection { $\theta' = 0$}

 We introduce notations $\theta(x_d) = \theta_d$, $\theta'(x_d)=0$: 

$$\theta'(x_d)^2 = k \theta_d + 8 j_x \theta_d^{\frac{3}{2}} + 8 j_x \theta_d^{\frac{1}{2}}  + 4\beta^2 = 0$$

\begin{equation}\label{eqthetad}
    k  \theta_d + 8 j_x \theta_d^{\frac{3}{2}} + 8 j_x \theta_d^{\frac{1}{2}}  + 4\beta^2 = 0.              \hspace{1.5cm}
\end{equation}
Our goal to solve Eq. \eqref{eqthetad} and consider the bifurcation problem of solutions depending on the parameters $j_{x}$, $\beta$, $k$ and boundary conditions on anode. 
Substitution $\Theta_d = u^2$ gives the following third order polynomial:

$$8 j_x u^3 + k u^2 +  8 j_x u  + 4\beta^2 = 0. $$
Notations 
$$\hat{k} = \dfrac{k}{8 j_x} \,\,\,\,\,\,\,\,,\,\,\,\,\,\,\, \hat{\beta} = \dfrac{4 \beta^2}{8 j_x}  \,\,\,\,\,\,\,\,,\,\,\,\,\,\,\, j_x \neq 0$$
 reduces the cubic equation to the form 

\begin{equation}\label{cubic}
    u^3 + \hat{k} \, u^2 +  u  + \hat{\beta} = 0.
\end{equation}

\subsubsection{Cubic equation}

We know by the Galois group theory \cite{Stewart2015} that any polynomial of grade $n$ lower than 4 has $n$ complex solutions. In our case, we want to find  three analytical solutions related to the equation \eqref{cubic}.

We  define the discriminant of  equation:
$$\Delta_u = 18\hat{k}\hat{\beta} + \hat{k}^2 - 4 - 4\hat{k}^3\hat{\beta} - 27 \hat{\beta}^2.$$

\begin{prop} \label{u_solutions_prop}

The cubic equation $u^3 + \hat{k} \, u^2 +  u  + \hat{\beta} = 0$ has the following  solutions in $\mathbb{C}$:

\begin{enumerate}
    \item $\Delta_u < 0$ :

    $$\begin{array}{rcl}
        u_1 & =  & -\dfrac{\hat{k}}{3} + \dfrac{\sqrt[3]{4}}{18} \left( \sqrt[3]{A_1 + A_2} + \sqrt[3]{A_1 - A_2} \right) \\
    \end{array}$$

    $$\begin{array}{rcl}
        u_2 & = & \left[ - \dfrac{\hat{k}}{3}  - \dfrac{\sqrt[3]{4}}{36} \left( \sqrt[3]{A_1 + A_2} + \sqrt[3]{A_1 - A_2} \right) \right] + \\ \\
        & & \left[ \dfrac{\sqrt{3}\sqrt[3]{4}}{36} \left( \sqrt[3]{A_1 + A_2} - \sqrt[3]{A_1 - A_2} \right) \right]i
    \end{array}$$

    $$\begin{array}{rcl}
        u_3 &=& \left[ - \dfrac{\hat{k}}{3}  - \dfrac{\sqrt[3]{4}}{36} \left( \sqrt[3]{A_1 + A_2} + \sqrt[3]{A_1 - A_2} \right) \right] + \\  \\
        & &\left[ - \dfrac{\sqrt{3}\sqrt[3]{4}}{36} \left( \sqrt[3]{A_1 + A_2} - \sqrt[3]{A_1 - A_2} \right) \right]i \\
    \end{array}$$
    
    $$\begin{array}{rcl}
        A_1 &=& -\left( 54\hat{k}^3 - 243\hat{k} + 729\hat{\beta} \right) \\  \\
        A_2& = &\sqrt{\left( 54\hat{k}^3 - 243\hat{k} + 729\hat{\beta} \right)^2 + 2916 \left(3 - \hat{k}^2\right)^3 } \\
    \end{array}$$
         
    \item  $\Delta_u = 0$ , $\hat{\beta} = \pm \dfrac{\sqrt{3}}{9}$,  $\hat{k}=\pm \sqrt{3}$:

    $$\begin{array}{rcl}
        u &=& \mp \dfrac{\sqrt{3}}{3} \\
    \end{array}$$

    \item  $\Delta_u = 0$ , $\hat{\beta} \neq \pm \dfrac{\sqrt{3}}{9}$,  $\hat{k}\neq\pm \sqrt{3}$:

        $$
        \begin{array}{rcl}
            u_1 &=&\dfrac{\hat{k}^3 - 4\hat{k} + 9\hat{\beta}}{3 - \hat{k}^2}\\ \\
            u_2 &=& \dfrac{- \hat{k} + 9\hat{\beta}}{2\hat{k}^2 - 6} \\
        \end{array}
        $$

    \item  $\Delta_u > 0$:
    $$\begin{array}{rcl}
        u_1 &=& A_3 \, \cos{\left( \dfrac{1}{3}\arccos{\left(  A_4 \right)} \right)} - \dfrac{\hat{k}}{3}\\
        u_2 &=& A_3 \, \cos{\left( \dfrac{1}{3}\arccos{\left(  A_4 \right)} + \dfrac{2\pi}{3} \right)} - \dfrac{\hat{k}}{3}\\
        u_3 &=& A_3 \, \cos{\left( \dfrac{1}{3}\arccos{\left(  A_4 \right)} + \dfrac{4\pi}{3}\right)} - \dfrac{\hat{k}}{3} \\
    \end{array}$$

    $$\begin{array}{rcl}
        A_3 &=& \dfrac{2}{3}\sqrt{\hat{k}^2-3} \\ \\
        A_4 &=& \dfrac{4\hat{k}^3-9\hat{k}+27\hat{\beta}}{18 - 6\hat{k}^2} \sqrt{\dfrac{9}{\hat{k}^2-3}}
    \end{array}$$
\end{enumerate}

\end{prop}

\begin{proof}

Case 1: $\Delta_{u} < 0$. We apply Tschirnhaus transformation \cite{Cardano1545}  $u = y - \frac{\hat{k}}{3}$ to \eqref{cubic} in order to get the depressed form of the cubic equation:

\begin{equation}\label{cubictrans}
y^3 + \left( \dfrac{3 - \hat{k}^2}{3}\right) y  + \left(\dfrac{2\hat{k}^3 - 9\hat{k} + 27\hat{\beta}}{27} \right) = 0.
\end{equation}
We  apply Cardano's rule to present (57) as an equation of shape $(m+n)^3 - 3mn(m+n)-(m^3+n^3)=0\;$\cite{Cardano1545}. In our case:

\begin{equation}\label{conditions}
    y = m+n \quad ,\quad -\dfrac{3 - \hat{k}^2}{9} = mn \quad , \quad-\dfrac{2\hat{k}^3 - 9\hat{k} + 27\hat{\beta}}{27} = m^3 + n^3.
\end{equation}

With  inter-medium result, we can apply some algebraic operations to obtain a six-order equation to solve for the value $n$:

$$729n^6 +\left( 54\hat{k}^3 - 243\hat{k} + 729\hat{\beta} \right)n^3 - \left(3 - \hat{k}^2\right)^3 = 0.$$

We  obtain a first set of solutions by solving  equation for $n^3$. 

$$n^3 = \dfrac{1}{1458}\left(-\left( 54\hat{k}^3 - 243\hat{k} + 729\hat{\beta} \right) \pm \sqrt{\left( 54\hat{k}^3 - 243\hat{k} + 729\hat{\beta} \right)^2 + 2916 \left(3 - \hat{k}^2\right)^3 }\right)$$

Because we want to guarantee a real solution for this first case (the complex results will be treated using the cube unity roots), we will impose a restriction for  quadratic discriminant.
$$
\begin{array}{c}
    \left( 54\hat{k}^3 - 243\hat{k} + 729\hat{\beta} \right)^2 + 2916 \left(3 - \hat{k}^2\right)^3 > 0 \\
    -19638\left(18\hat{k}\hat{\beta} + \hat{k}^2 - 4 - 4\hat{k}^3\hat{\beta} - 27 \hat{\beta}^2\right) > 0 \\
    -19638 \Delta_u >0\\
    \Delta_u < 0.
\end{array}
$$

As the real solutions for $n^3$ depends on $\Delta_u < 0$, we will set this as a restriction for the first analytical approach of the problem. The other cases will be set afterwards.

Given the solution for $n^3$, we can use the relationships set between $m$ and $n$ in \eqref{conditions} and some algebraic operations to find their values:

\begin{equation}\label{m and n}
\scriptsize
\begin{array}{cl}
m = \dfrac{\sqrt[3]{4}}{18}\sqrt[3]{ -\left( 54\hat{k}^3 - 243\hat{k} + 729\hat{\beta} \right) - \sqrt{\left( 54\hat{k}^3 - 243\hat{k} + 729\hat{\beta} \right)^2 + 2916 \left(3 - \hat{k}^2\right)^3 }} & \\
n = \dfrac{\sqrt[3]{4}}{18}\sqrt[3]{ -\left( 54\hat{k}^3 - 243\hat{k} + 729\hat{\beta} \right) + \sqrt{\left( 54\hat{k}^3 - 243\hat{k} + 729\hat{\beta} \right)^2 + 2916 \left(3 - \hat{k}^2\right)^3 }}&
\end{array}
\end{equation}

With Cardano's rule \cite{Cardano1545}, we  define  three complex solutions for  equation \eqref{cubictrans} when $\Delta_u < 0$ using \eqref{m and n} :

$$y_1 = m + n$$
$$y_2 = - \dfrac{1}{2}(m + n) + \dfrac{\sqrt{3}}{2}(m-n)i$$
$$y_3 = - \dfrac{1}{2}(m + n) - \dfrac{\sqrt{3}}{2}(m-n)i$$

Now we apply  the inverse Tschirnhaus transformation to our $y$ solutions. We  define the coefficients $A_1$ and $A_2$:

$$
A_1 = -\left( 54\hat{k}^3 - 243\hat{k} + 729\hat{\beta} \right)
$$
$$A_2 = \sqrt{\left( 54\hat{k}^3 - 243\hat{k} + 729\hat{\beta} \right)^2 + 2916 \left(3 - \hat{k}^2\right)^3 }
$$
Then, the solutions for \eqref{cubic} when $\Delta_u < 0$ are:

\begin{equation}\label{soldelta<0}
    \begin{array}{cl}
        u_1 = -\dfrac{\hat{k}}{3} + \dfrac{\sqrt[3]{4}}{18} \left( \sqrt[3]{A_1 + A_2} + \sqrt[3]{A_1 - A_2} \right)\\
        u_2 = \left[ - \dfrac{\hat{k}}{3}  - \dfrac{\sqrt[3]{4}}{36} \left( \sqrt[3]{A_1 + A_2} + \sqrt[3]{A_1 - A_2} \right) \right] + \left[ \dfrac{\sqrt{3}\sqrt[3]{4}}{36} \left( \sqrt[3]{A_1 + A_2} - \sqrt[3]{A_1 - A_2} \right) \right]i\\
        u_3 = \left[ - \dfrac{\hat{k}}{3}  - \dfrac{\sqrt[3]{4}}{36} \left( \sqrt[3]{A_1 + A_2} + \sqrt[3]{A_1 - A_2} \right) \right] + \left[ - \dfrac{\sqrt{3}\sqrt[3]{4}}{36} \left( \sqrt[3]{A_1 + A_2} - \sqrt[3]{A_1 - A_2} \right) \right]i
    \end{array}
\end{equation}
\\

\noindent\textbf{Case 2-3: $\Delta_u = 0$}\\
In this case, we have 

$$u^3 + \hat{k} \, u^2 +  u  + \hat{\beta} = (u - u_1)(u - u_2)^2$$
and corresponding depressed equation \eqref{cubictrans} as follows:

$$y^3 + \left( \dfrac{3 - \hat{k}^2}{3}\right) y  + \left(\dfrac{2\hat{k}^3 - 9\hat{k} + 27\hat{\beta}}{27} \right) = (y-y_1)(y-y_2)^2.  $$
 
As we expand this expression, we get the following non-linear equations system to solve:
$$
\begin{array}{rcl}
 -2y_2 -y_1 &=& 0 \\
 y_2^2 + 2y_1y_2 &=& \dfrac{3 - \hat{k}^2}{3} \\
-y_2^2y_1 &=& \dfrac{2\hat{k}^3 - 9\hat{k} + 27\hat{\beta}}{27} \\
\end{array}$$

After we use some simple algebraic techniques to manipulate the system, we get the following results:

$$y_1 =  \dfrac{2\hat{k}^3 - 9\hat{k} + 27\hat{\beta}}{9 - 3\hat{k}^2}$$
$$y_2  =  \dfrac{2\hat{k}^3 - 9\hat{k} + 27\hat{\beta}}{6\hat{k}^2 - 18}$$

Then we apply the inverse Tschirnhaus transformation to our $y$ solutions in order to get the \eqref{cubic} solutions:

\begin{equation}\label{soldelta=0_1}
    \begin{array}{cl}
        u_1 =\dfrac{\hat{k}^3 - 4\hat{k} + 9\hat{\beta}}{3 - \hat{k}^2}\\
        u_2 = \dfrac{- \hat{k} + 9\hat{\beta}}{2\hat{k}^2 - 6} \\
    \end{array}
\end{equation}

In the case we have a single real solution with multiplicity $3$, we have the following system:

$$u^3 + \hat{k} \, u^2 +  u  + \hat{\beta} = (u - u_1)^3$$

As we expand this expression, we get the following non-linear equations system to solve:

$$
\begin{array}{rcl}
 3u_1 &=& -\hat{k} \\
 3u_1^2 &=& 1 \\
u_1^3 &=& -\hat{\beta} \\
\end{array}$$

As there is only one variable with three equations in the system, it will result that this particular case only can happen with an specific set of values for the parameters $\hat{k}$ and $\hat{\beta}$. The solution for \eqref{cubic} is:

\begin{equation}\label{soldelta=0_2}
    \hat{\beta} = \pm \dfrac{\sqrt{3}}{9} \,\, \wedge \,\, \hat{k}=\pm \sqrt{3} \,\, \Longrightarrow \,\, u_1 = \mp \dfrac{\sqrt{3}}{3}
\end{equation}
\\

\noindent\textbf{Case 4: $\Delta_u > 0$}\\

For this case, we want to find a set of three solutions where all of them are real (imaginary part equal to $0$) and different from each other. However, Galois theory \cite{Stewart2015} shows us that this case cannot be expressed as real radicals. So in this case, we will apply a different approach.

First, we reduce to the depressed cubic equation described by \eqref{cubictrans} and then we will apply the Viète's solutions \cite{Viete1591}; these say that given a cubic equation of shape $z^3 + pz + q = 0$, we can express its three solutions as:

$$z_{t+1} = 2\sqrt{-\dfrac{p}{3}} \text{cos}\left( \dfrac{1}{3} \text{arccos}\left( \dfrac{3q}{2p} \sqrt{\dfrac{-3}{p}}\right) + \dfrac{2t\pi}{3}\right)\,\,\,\, \forall t \in \{0,1,2\}$$

Applying this definition to our problem, we get as solutions for the depressed equation:

$$y_{1} = \dfrac{2}{3}\sqrt{\hat{k}^2-3} \cos{\left( \dfrac{1}{3} \arccos{\left( \dfrac{2\hat{k}^3-9\hat{k}+27\hat{\beta}}{18 - 6\hat{k}^2} \sqrt{\dfrac{9}{\hat{k}^2-3}}\right)}\right)}$$

$$y_{2} = \dfrac{2}{3}\sqrt{\hat{k}^2-3} \cos{\left( \dfrac{1}{3} \arccos{\left( \dfrac{2\hat{k}^3-9\hat{k}+27\hat{\beta}}{18 - 6\hat{k}^2} \sqrt{\dfrac{9}{\hat{k}^2-3}}\right)+\dfrac{2\pi}{3}}\right)}$$

$$y_{3} = \dfrac{2}{3}\sqrt{\hat{k}^2-3} \cos{\left( \dfrac{1}{3} \arccos{\left( \dfrac{2\hat{k}^3-9\hat{k}+27\hat{\beta}}{18 - 6\hat{k}^2} \sqrt{\dfrac{9}{\hat{k}^2-3}}\right)+\dfrac{4\pi}{3}}\right)}$$

Then, as we want to define the solutions for \eqref{cubic}, we need to apply the inverse Tschirnhaus transformation to our $y$ solutions.

We will define the quantities $A_3$ and $A_4$ in order to show in a more readable format the results:

$$A_3 = \dfrac{2}{3}\sqrt{\hat{k}^2-3}$$
$$A_4 = \dfrac{2\hat{k}^3-9\hat{k}+27\hat{\beta}}{18 - 6\hat{k}^2} \sqrt{\dfrac{9}{\hat{k}^2-3}}$$

The final solutions for this case are:

\begin{equation}\label{soldelta>0}
    \begin{array}{cl}
        u_1 = A_3 \, \cos{\left( \dfrac{1}{3}\arccos{\left(  A_4 \right)} \right)} - \dfrac{\hat{k}}{3} & \\
        u_2 = A_3 \, \cos{\left( \dfrac{1}{3}\arccos{\left(  A_4 \right)} + \dfrac{2\pi}{3} \right)} - \dfrac{\hat{k}}{3} & \\
        u_3 = A_3 \, \cos{\left( \dfrac{1}{3}\arccos{\left(  A_4 \right)} + \dfrac{4\pi}{3}\right)} - \dfrac{\hat{k}}{3} & \\
    \end{array}
\end{equation}

With this, we have already covered all the cases related to the cubic equation. Then, we can join all this processes to give the general solution of \eqref{cubic}.

\end{proof}

\subsubsection{Effective potential $\Theta_d$}

We used the substitution $\Theta_d = u^2$ to transform  non-linear equation to a cubic equation, therefore,  we have to apply the same resource to bring the set of solutions for the variable $\Theta_d$.

However, we should be careful as we need to guarantee that beside the transformation, the solution might be valid for $u$ and $\Theta_d$. For this process, let's think about a value $\bar{u}$ which is the solution of the cubic equation for a certain paramaters $\bar{k}$ and $\bar{\beta}$. Let's think this as a function that satisfies the following:

$$p(u) = u^3 + \bar{k} \, u^2 +  u  + \bar{\beta} \Longrightarrow p(\bar{u})=0$$

We define the sign of  value $\bar{u}$. For that, we will use the definition of an even function:  $f(x)=f(-x)$.

A function $p(u)$ does not satisfies the even condition because it is a polynomial of odd degree.  In fact, we can show that the definition only satisfies for $u=0$, $u=i$ and $u=-i$.

We can infer:

$$\bar{u} \notin \{i,-i,0\} \Longrightarrow p(\bar{u})=0 \,\, \wedge \,\, p(-\bar{u})\neq0$$

We will check what happens with the results when we transform them from  $u$ to  $\Theta_d$ domain. First we must know that for a solution to be valid in our problem, we need to guarantee that when it is transformed from a domain to another (Does not matter in which order is done first), it is still a solution for both of them. As this transformation does not have a valid inverse for all values, we must restrict our solutions set:

$$\Re(u) \geq 0 \,\, \Longrightarrow \,\, \Theta_d = u^ 2$$

If $\Re(u)<0$, we can guarantee that it is a solution for \eqref{cubic}, but we cannot do it for \eqref{eqthetad} , even if we try manipulating the sign of $u$, we cannot guarantee it because of the statement proofed before. As a result, our position for this problem is to discard those solutions that do not match the condition.

Then, we will set the new condition and the transformation for the general solution given for the cubic equation. We will check each case set by the cubic discriminant.\\

\noindent\textbf{Case 1: $\Delta_u < 0$}\\

In this case where we have complex conjugated solutions, we must set a condition over the real part of these solutions and the real solution.

In the case where:

$$\dfrac{6\hat{k}}{\sqrt[3]{4}} < \sqrt[3]{A_1 + A_2} + \sqrt[3]{A_1 - A_2}$$

we can say that:

$$\begin{array}{rcl}
    \Theta_{d1} &= &\dfrac{\hat{k}^{2}}{9} - \dfrac{\sqrt[3]{4}\hat{k}}{27}\left( \sqrt[3]{A_1 + A_2} + \sqrt[3]{A_1 - A_2}\right)+ \\
    & &\dfrac{\sqrt[3]{2}}{162} \left(\sqrt[3]{(A_1 + A_2)^2} + \sqrt[3]{(A_1 - A_2)^2} -18 \sqrt[3]{4} \left(3-\hat{k}^2\right)\right)
\end{array}$$

$$$$

And in the case where:

$$-\dfrac{12\hat{k}}{\sqrt[3]{4}} > \sqrt[3]{A_1 + A_2} + \sqrt[3]{A_1 - A_2}$$

we define $t_R$ and $t_I$ as:

$$\begin{array}{rcl}
    t_R &=& \dfrac{\hat{k}^{2}}{9}+\dfrac{\sqrt[3]{4}\hat{k}}{54}\left(\sqrt[3]{A_1 + A_2} + \sqrt[3]{A_1 - A_2} \right) - \\ \\
    & &\dfrac{\sqrt[3]{2}}{324} \left( \sqrt[3]{(A_1 + A_2)^{2}} + \sqrt[3]{(A_1 - A_2)^{2}} \right) - \dfrac{2}{9} \left( 3 - \hat{k}^{2} \right) \\ \\
    t_I &=& -\dfrac{\sqrt{3}\sqrt[3]{4}\hat{k}}{54} \left( \sqrt[3]{A_1 + A_2} - \sqrt[3]{A_1 - A_2} \right) - \\ \\
    & & \dfrac{\sqrt{3} \sqrt[3]{2}}{324} \left( \sqrt[3]{(A_1 + A_2)^{2}} - \sqrt[3]{(A_1 - A_2)^{2}} \right) \\
\end{array}$$

and the solutions are:

$$\Theta_{d2} = t_R - t_I i$$
$$\Theta_{d3} = t_R + t_I i$$\\

\noindent\textbf{Case 2: $\Delta_u = 0$}\\

In this case as solutions do not have an imaginary part, we will set our restrictions over the complete solutions.

First we will see the case where one solution has multiplicity of two:

If we have that:

$$\dfrac{\hat{k}^3 - 4\hat{k} + 9\hat{\beta}}{3 - \hat{k}^2} > 0$$

then we know that:

$$\Theta_{d1} = \dfrac{\left(2\hat{k}^3 - 9\hat{k} + 27\hat{\beta} \right)^2}{81 - 54\hat{k}^2 + 9\hat{k}^4} - \dfrac{4\hat{k}^4 - 18\hat{k}^2+ 54\hat{k}\hat{\beta}}{27-9\hat{k}^2} + \dfrac{\hat{k}^2}{9}$$

If we have that:

$$\dfrac{- \hat{k} + 9\hat{\beta}}{2\hat{k}^2 - 6} > 0$$

then we know that:

$$\Theta_{d2} =  \dfrac{\left(2\hat{k}^3 - 9\hat{k} + 27\hat{\beta} \right)^2}{36\hat{k}^4-216\hat{k}^2+324} - \dfrac{2\hat{k}^4 - 9\hat{k}^2 + 27\hat{k}\hat{\beta}}{9\hat{k}^2-27} + \dfrac{\hat{k}^2}{9}$$\\

In the case where there is a single root with multiplicity of $3$, we must check if $\hat{\beta} =  - \frac{\sqrt{3}}{9}$ and $\hat{k} = - \sqrt{3}$. In that case: 

$$\Theta_{d} = \dfrac{1}{3}$$\\

\noindent\textbf{Case 3: $\Delta_u > 0$}\\

In this case as we have three real solutions, we restrict our problem with the value of $u$ as we did before:\\

If we now that:

$$\text{arccos}\left( \dfrac{\hat{k}}{2\sqrt{\hat{k}^2-3}}\right) >  \dfrac{1}{3} \text{arccos}\left( \dfrac{2\hat{k}^3-9\hat{k}+27\hat{\beta}}{18 - 6\hat{k}^2} \sqrt{\dfrac{9}{\hat{k}^2-3}}\right)$$

then the first solution will be:

$$\Theta_{d1} =\left( \dfrac{2}{3}\sqrt{\hat{k}^2-3}\,\,\,\text{cos}\left( \dfrac{1}{3} \text{arccos}\left( \dfrac{2\hat{k}^3-9\hat{k}+27\hat{\beta}}{18 - 6\hat{k}^2} \sqrt{\dfrac{9}{\hat{k}^2-3}}\right)\right) - \dfrac{\hat{k}}{3} \right)^2$$

If we now that:

$$\text{arccos}\left( \dfrac{\hat{k}}{2\sqrt{\hat{k}^2-3}}\right) >  \dfrac{1}{3} \text{arccos}\left( \dfrac{4\hat{k}^3-9\hat{k}+27\hat{\beta}}{18 - 6\hat{k}^2} \sqrt{\dfrac{9}{\hat{k}^2-3}}\right) + \dfrac{2\pi}{3}$$

then the second solution will be:

$$\Theta_{d2} = \left(\dfrac{2}{3}\sqrt{\hat{k}^2-3}\,\,\, \text{cos}\left( \dfrac{1}{3} \text{arccos}\left( \dfrac{2\hat{k}^3-9\hat{k}+27\hat{\beta}}{18 - 6\hat{k}^2} \sqrt{\dfrac{9}{\hat{k}^2-3}}\right) + \dfrac{2\pi}{3}\right) - \dfrac{\hat{k}}{3}\right)^2$$

If we now that:

$$\text{arccos}\left( \dfrac{\hat{k}}{2\sqrt{\hat{k}^2-3}}\right) >  \dfrac{1}{3} \text{arccos}\left( \dfrac{2\hat{k}^3-9\hat{k}+27\hat{\beta}}{18 - 6\hat{k}^2} \sqrt{\dfrac{9}{\hat{k}^2-3}}\right) + \dfrac{4\pi}{3}$$

then the third solution will be:

$$\Theta_{d3} =\left( \dfrac{2}{3}\sqrt{\hat{k}^2-3}\,\,\, \text{cos}\left( \dfrac{1}{3} \text{arccos}\left( \dfrac{2\hat{k}^3-9\hat{k}+27\hat{\beta}}{18 - 6\hat{k}^2} \sqrt{\dfrac{9}{\hat{k}^2-3}}\right) + \dfrac{4\pi}{3}\right) - \dfrac{\hat{k}}{3}\right)^2$$

Now with all these new restrictions over the parameters, we can define a general formula for the analytical solution of $\Theta_d$.

\begin{prop}\label{prop_theta_solutions}  We define the set of parameters:
$$\hat{k} = \dfrac{k}{8 j_x} \,\,\,\,\,\,\,\,,\,\,\,\,\,\,\, \hat{\beta} = \dfrac{4 \beta^2}{8 j_x}  \,\,\,\,\,\,\,\,,\,\,\,\,\,\,\, j_x \neq 0$$
and the discriminant $\Delta_u = 18\hat{k}\hat{\beta} + \hat{k}^2 - 4 - 4\hat{k}^3\hat{\beta} - 27 \hat{\beta}^2$ for the non-linear equation $k  \Theta_d + 8 j_x \Theta_d^{\frac{3}{2}} + 8 j_x \Theta_d^{\frac{1}{2}}  + 4\beta^2 = 0$.

Then its solutions in $\mathbb{C}$ are:

\begin{enumerate}
    \item If $\Delta_u < 0$ we define:
    \begin{equation*}
    \begin{array}{c}
            A_1 = -\left( 54\hat{k}^3 - 243\hat{k} + 729\hat{\beta} \right) \\
            A_2 = \sqrt{\left( 54\hat{k}^3 - 243\hat{k} + 729\hat{\beta} \right)^2 + 2916 \left(3 - \hat{k}^2\right)^3 } \\
    \end{array}
    \end{equation*}
    \begin{equation*}
    \scriptsize
    \begin{array}{c}
            t_R = \dfrac{\hat{k}^{2}}{9}+\dfrac{\sqrt[3]{4}\hat{k}}{54}\left(\sqrt[3]{A_1 + A_2} + \sqrt[3]{A_1 - A_2} \right) - \dfrac{\sqrt[3]{2}}{324} \left( \sqrt[3]{(A_1 + A_2)^{2}} + \sqrt[3]{(A_1 - A_2)^{2}} \right) - \dfrac{2}{9} \left( 3 - \hat{k}^{2} \right) \\
            t_I = -\dfrac{\sqrt{3}\sqrt[3]{4}\hat{k}}{54} \left( \sqrt[3]{A_1 + A_2} - \sqrt[3]{A_1 - A_2} \right) - \dfrac{\sqrt{3} \sqrt[3]{2}}{324} \left( \sqrt[3]{(A_1 + A_2)^{2}} - \sqrt[3]{(A_1 - A_2)^{2}} \right) 
    \end{array}
    \end{equation*}
            
    and the solutions are:

    \begin{itemize}
        \item If the following is satisfied:

        $$\dfrac{6\hat{k}}{\sqrt[3]{4}} < \sqrt[3]{A_1 + A_2} + \sqrt[3]{A_1 - A_2}$$
        
        then:
        
        $$\begin{array}{rcl}
            \Theta_{d1} &=& \dfrac{\hat{k}^{2}}{9} - \dfrac{\sqrt[3]{4}\hat{k}}{27}\left( \sqrt[3]{A_1 + A_2} + \sqrt[3]{A_1 - A_2}\right)\\ & &+\dfrac{\sqrt[3]{2}}{162} \left(\sqrt[3]{(A_1 + A_2)^2} + \sqrt[3]{(A_1 - A_2)^2} -18 \sqrt[3]{4} \left(3-\hat{k}^2\right)\right)
        \end{array}$$
        
        \item If the following is satisfied:
        
        $$-\dfrac{12\hat{k}}{\sqrt[3]{4}} > \sqrt[3]{A_1 + A_2} + \sqrt[3]{A_1 - A_2}$$
        
        then:
        $$\Theta_{d2} = t_R - t_I i$$
        $$\Theta_{d3} = t_R + t_I i$$
    \end{itemize}
          
    \item If $\Delta_u = 0$ , $\hat{\beta} = - \dfrac{\sqrt{3}}{9}$ and $\hat{k}=- \sqrt{3}$, the solution is:

        $$\Theta_d = \dfrac{1}{3}$$

    \item If $\Delta_u = 0$ , $\hat{\beta} \neq - \dfrac{\sqrt{3}}{9}$ and $\hat{k}\neq- \sqrt{3}$
    the solutions are:

        \begin{itemize}
        \item If the following is satisfied:

        $$\dfrac{\hat{k}^3 - 4\hat{k} + 9\hat{\beta}}{3 - \hat{k}^2} > 0$$

        then:

        $$\Theta_{d1} = \dfrac{\left(2\hat{k}^3 - 9\hat{k} + 27\hat{\beta} \right)^2}{81 - 54\hat{k}^2 + 9\hat{k}^4} - \dfrac{4\hat{k}^4 - 18\hat{k}^2+ 54\hat{k}\hat{\beta}}{27-9\hat{k}^2} + \dfrac{\hat{k}^2}{9}$$

        \item If the following is satisfied:
        
            $$\dfrac{- \hat{k} + 9\hat{\beta}}{2\hat{k}^2 - 6} > 0$$
            
            then:
            
            $$\Theta_{d2} =  \dfrac{\left(2\hat{k}^3 - 9\hat{k} + 27\hat{\beta} \right)^2}{36\hat{k}^4-216\hat{k}^2+324} - \dfrac{2\hat{k}^4 - 9\hat{k}^2 + 27\hat{k}\hat{\beta}}{9\hat{k}^2-27} + \dfrac{\hat{k}^2}{9}$$

        \end{itemize}

    \item If $\Delta_u > 0$ we define:
            $$A_3 = \dfrac{2}{3}\sqrt{\hat{k}^2-3}$$
            $$A_4 = \dfrac{2\hat{k}^3-9\hat{k}+27\hat{\beta}}{18 - 6\hat{k}^2} \sqrt{\dfrac{9}{\hat{k}^2-3}}$$
          and the solutions are:
          \begin{itemize}
          \item If the following is satisfied:
          
          $$\text{arccos}\left( \dfrac{\hat{k}}{2\sqrt{\hat{k}^2-3}}\right) <  \dfrac{1}{3} \text{arccos}\left( \dfrac{2\hat{k}^3-9\hat{k}+27\hat{\beta}}{18 - 6\hat{k}^2} \sqrt{\dfrac{9}{\hat{k}^2-3}}\right)$$
          
          then:
          $$
            \begin{array}{cl}
                \Theta_{d1} = \left(A_3 \, \cos{\left( \dfrac{1}{3}\arccos{\left(  A_4 \right)} \right)} - \dfrac{\hat{k}}{3}\right)^2 & \\
            \end{array}
          $$
          \item If the following is satisfied:
          
            $$\text{arccos}\left( \dfrac{\hat{k}}{2\sqrt{\hat{k}^2-3}}\right) <  \dfrac{1}{3} \text{arccos}\left( \dfrac{2\hat{k}^3-9\hat{k}+27\hat{\beta}}{18 - 6\hat{k}^2} \sqrt{\dfrac{9}{\hat{k}^2-3}}\right) + \dfrac{2\pi}{3}$$
          
          then:
          $$
            \begin{array}{cl}
                \Theta_{d2} = \left(A_3 \, \cos{\left( \dfrac{1}{3}\arccos{\left(  A_4 \right)} + \dfrac{2\pi}{3} \right)} - \dfrac{\hat{k}}{3}\right)^2 &
            \end{array}
          $$
          \item If the following is satisfied:
          
          $$\text{arccos}\left( \dfrac{\hat{k}}{2\sqrt{\hat{k}^2-3}}\right) <  \dfrac{1}{3} \text{arccos}\left( \dfrac{2\hat{k}^3-9\hat{k}+27\hat{\beta}}{18 - 6\hat{k}^2} \sqrt{\dfrac{9}{\hat{k}^2-3}}\right) + \dfrac{4\pi}{3}$$
          
          then:
          $$
            \begin{array}{cl}
                 \Theta_{d3} = \left(A_3 \, \cos{\left( \dfrac{1}{3}\arccos{\left(  A_4 \right)} + \dfrac{4\pi}{3}\right)} - \dfrac{\hat{k}}{3} \right)^2&
            \end{array}
          $$
          \end{itemize}
\end{enumerate}
\end{prop}

\section{Bifurcation Analysis}

The present section exposes a solution on the complex plane  of the cubic algebraic equation and equation with a non-integer degree which  are the complex solutions of equation for effective potential \eqref{theta2d}, this allows complex bifurcation diagrams display. A problem because the calculation of complex roots, is the increasing number of dimensions to project in bifurcation diagrams. This is because the number of state variables is doubled because of the increase of the imaginary part of these.  In addition, bifurcation diagrams with complex fixed points of several case studies are presented, in which solutions for certain values of the bifurcation parameters are shown.

 We have a way to define the general solution of the equations and it depends on a set of parameters; we can start a bifurcation analysis in order to understand how the solutions are behaving, and we are going to see them through computationally generated graphics.
 
\subsection{Bifurcation plots over $u$}

As it was shown in \textbf{Proposition 4.1}, we can plot the different solutions over the auxiliary variable $u$ space depending on the values of parameters $\hat{k}$ and $\hat{\beta}$ involved in the cubic equation. 

For our first example we plotted the solutions for $\hat{k}=-\sqrt{3}$ and $-5 \leq \hat{\beta} \leq 5$; also, we will split the graphics of the real and the imaginary part of the solutions.

\begin{figure}[h]
    \centering
    \includegraphics[width=\textwidth]{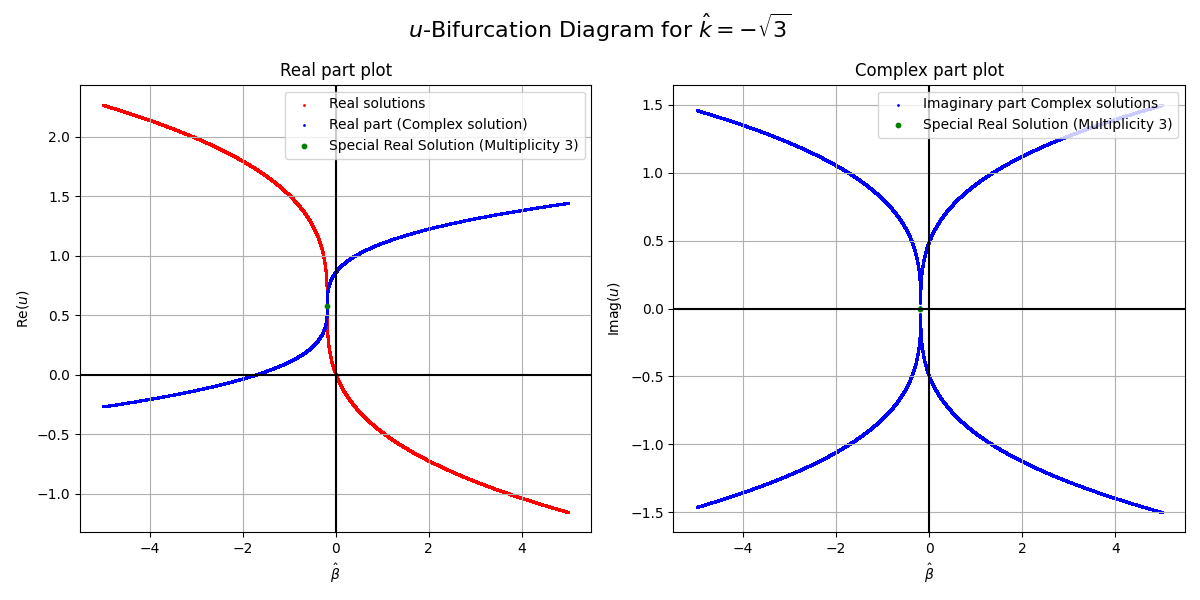}
    \caption{$u$ solutions given $\hat{k}=-\sqrt{3}$.}
    \label{bif_u_k_minus_sqrt_3}
\end{figure}

Figure \ref{bif_u_k_minus_sqrt_3} provides a detailed bifurcation diagram of the auxiliary variable \( u \) as a function of \( \hat{\beta} \), with the parameter \( \hat{k} \) fixed at \( -\sqrt{3} \). The plot is bifurcated into real (top) and imaginary (bottom) components. The real part plot reveals a single, continuous branch of real solutions (likely \( u_1 \)) and the real components of two complex-conjugate solution branches (\( u_2, u_3 \)), which form a distinct loop structure. The imaginary part plot confirms that the looping branches are purely complex outside the points of intersection with the real axis, where their imaginary parts vanish. The points where all three branches meet correspond to a discriminant \( \Delta_u = 0 \), indicating a transition in the solution multiplicity. This visualization clearly demarcates the parameter regions yielding one real versus one real and two complex solutions.

Following this idea, we plotted the solutions for $\hat{\beta}=-\frac{\sqrt{3}}{9}$ and $-5 \leq \hat{\beta} \leq 5$ with the same split done for the first scenario.

\begin{figure}[h]
    \centering
    \includegraphics[width=\textwidth]{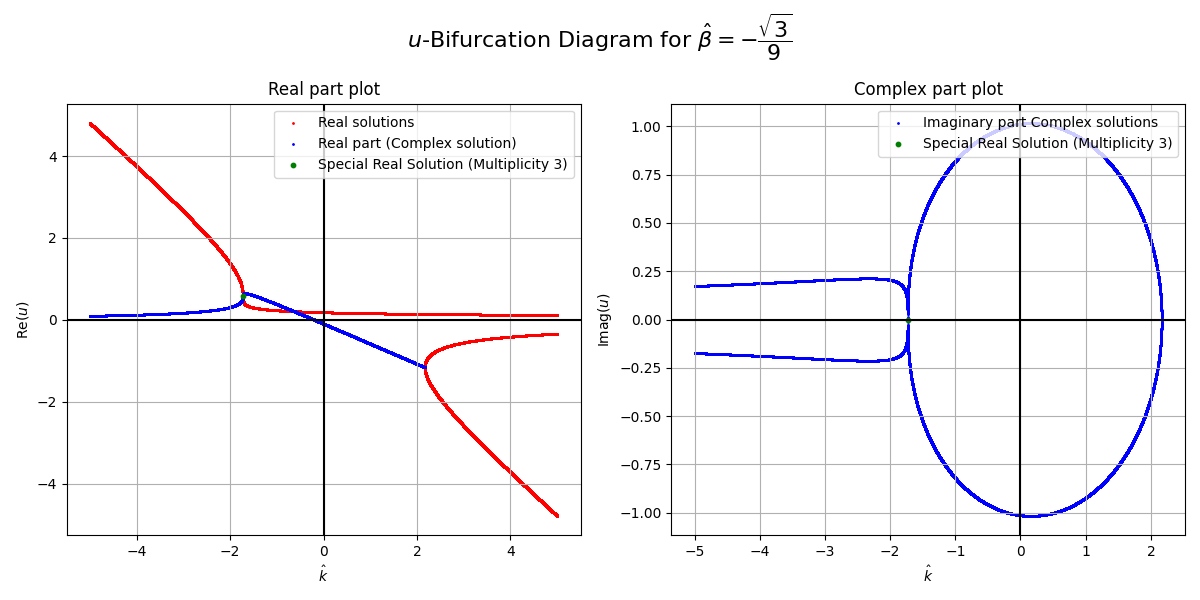}
    \caption{$u$ solutions given $\hat{\beta}=-\frac{\sqrt{3}}{9}$.}
    \label{bif_u_b_minus_sqrt_3_over_9}
\end{figure}

For the solutions described by figure \ref{bif_u_b_minus_sqrt_3_over_9}, we can see also a behavior related to the complex solutions with an imaginary part and the real solutions just where we see the multiplicity 3 real solution. It illustrates the bifurcation of \( u \) as \( \hat{k} \) varies, while \( \hat{\beta} \) is held constant at \( -\frac{\sqrt{3}}{9} \). Similar to figure \ref{bif_u_k_minus_sqrt_3}, the real and imaginary parts are displayed separately. A key feature is the presence of a single, multiplicity-three real solution at a specific \( \hat{k} \) value, manifesting as a point where all branches converge. Away from this point, the solution splits into one continuous real branch and a pair of complex-conjugate branches that form a closed loop in the parameter space. The loop's intersection with the real axis again signifies parameters for which the complex solutions become real (i.e., \( \Delta_u = 0 \)). This figure underscores how the solution structure is sensitive to variations in both parameters, not just \( \hat{\beta} \).

We can give even more views about the solutions bifurcation by fixing $\hat{k}$ and $\hat{\beta}$ to different possible values.

\begin{figure}[h]
    \centering
    \includegraphics[width=\textwidth]{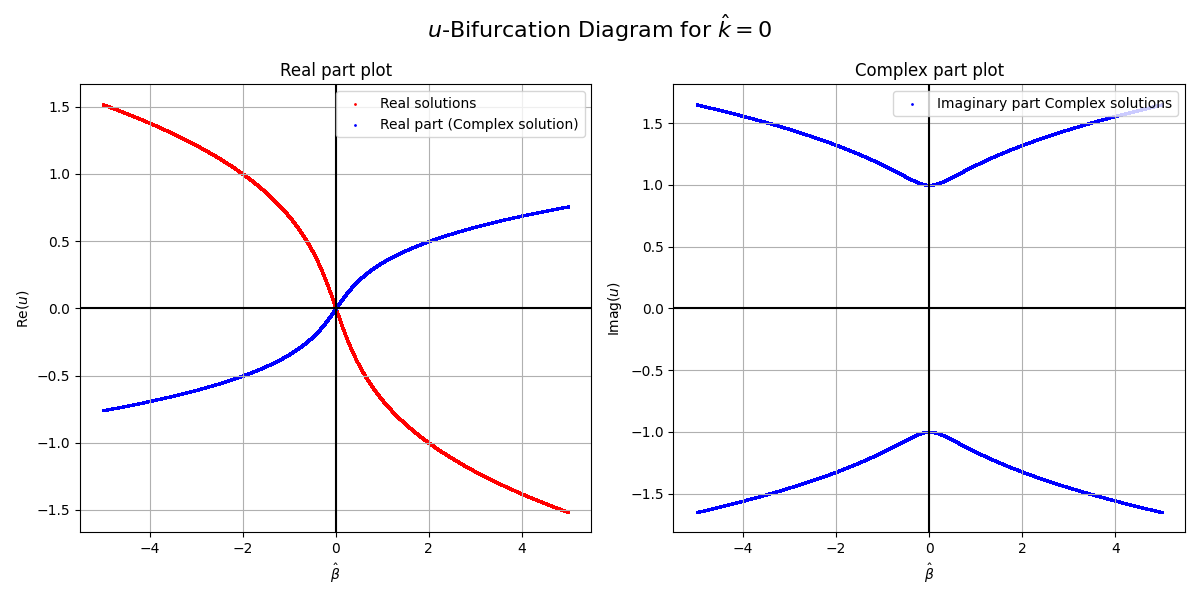}
    \caption{$u$ solutions given $\hat{k}=0$.}
    \label{bif_u_k_0}
\end{figure}

Figure \ref{bif_u_k_0} explores the solution space when \( \hat{\beta} = 0 \). Under this condition, the cubic equation simplifies, leading to a distinct bifurcation structure. The real part plot shows a primary real solution branch and the real components of the complex branches, which now exhibit a different, more open loop shape compared to Figures 1 and 2. The imaginary part plot confirms the non-zero imaginary components of the looping branches. This specific case helps isolate the influence of the \( \hat{k} \) parameter and demonstrates that the complex loop structures are a generic feature of the system's parameter space, not an artifact of a specific parameter choice.

\begin{figure}[h]
    \centering
    \includegraphics[width=\textwidth]{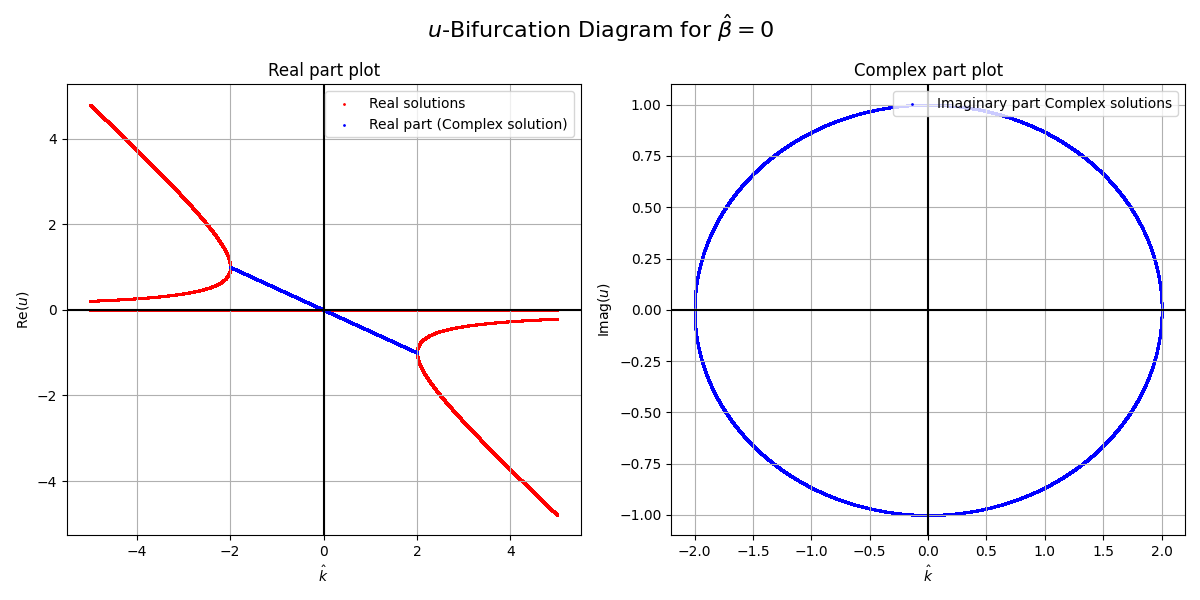}
    \caption{$u$ solutions given $\hat{\beta}=0$.}
    \label{bif_u_b_0}
\end{figure}

Figure \ref{bif_u_k_0} displays the bifurcation diagram for \( u \) with parameter \( \hat{k} \) set to zero. This configuration reveals a highly symmetric structure. The real part plot shows a single, odd-symmetric real solution passing through the origin, accompanied by the real parts of the complex branches that form a symmetric, figure-eight-like loop. The imaginary part plot is a perfect mirror image across the horizontal axis, reflecting the complex conjugate nature of the two non-real solutions. The points where the loop intersects the real axis are the bifurcation points where these complex solutions become real and distinct.

\begin{figure}[h]
    \centering
    \includegraphics[width=\textwidth]{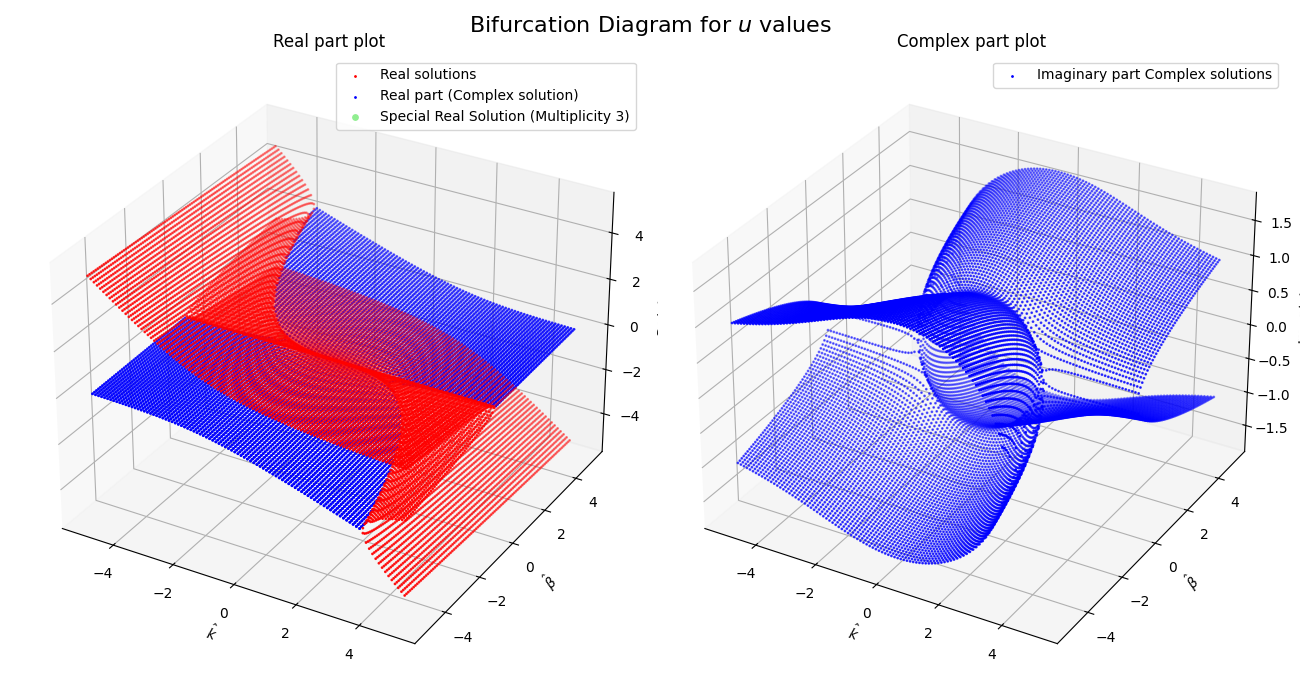}
    \caption{$u$ solutions surface.}
    \label{bif_u}
\end{figure}

In figure \ref{bif_u}, we showed the complex shape that the solutions surface acquires in the real and imaginary space. It presents a comprehensive, three-dimensional visualization of the entire \( u \)-solution surface, where the parameters \( \hat{k} \) and \( \hat{\beta} \) form the base plane, and the real part of \( u \) is plotted vertically. This global view synthesizes the behaviors observed in the previous cross-sectional diagrams. The surface exhibits folds and self-intersections, which are the geometric manifestations of bifurcations. Regions where the surface is single-valued correspond to a single real solution (\( \Delta_u < 0 \)), while regions where it folds over into three layers correspond to three distinct real solutions (\( \Delta_u > 0 \)). The edges of these folds are precisely the boundaries defined by \( \Delta_u = 0 \).

\subsection{Bifurcation plots over $\Theta$}

Following the transformation rules between $u$ and $\Theta$, we can now study the bifurcation plots in the $\Theta$ space. The main idea is to understand how the solutions obtained in the $u$ domain are projected into $\Theta$, and to explore how the restrictions applied during the transformation affect the admissible solutions.  

As in the $u$ case, we start by fixing $\hat{k}=-\sqrt{3}$ and varying $\hat{\beta}$ within the interval $[-5,5]$. This configuration lets us observe how the multiplicity-three solutions in $u$ translate into $\Theta$ and how the complex branches are mapped.  

\begin{figure}[h]
    \centering
    \includegraphics[width=\textwidth]{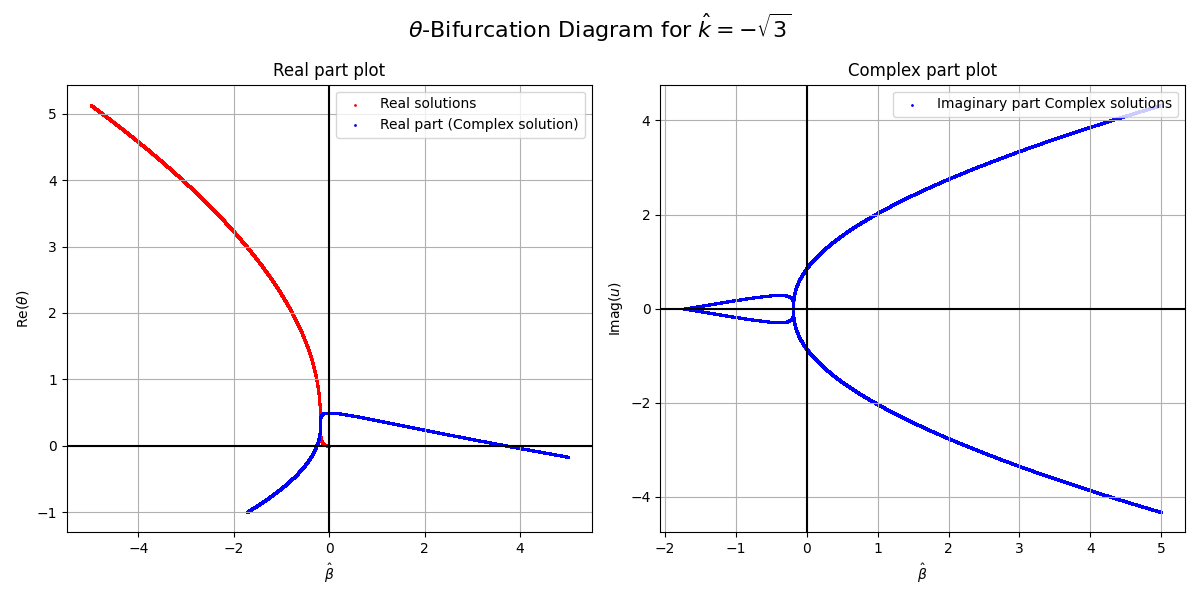}
    \caption{$\Theta$ solutions given $\hat{k}=-\sqrt{3}$.}
    \label{bif_theta_k_minus_sqrt_3}
\end{figure}

Figure \ref{bif_theta_k_minus_sqrt_3} shows the result of mapping the \( u \)-solutions from Figure \ref{bif_u_k_minus_sqrt_3} into the physical \( \Theta \)-space, with \( \hat{k} = -\sqrt{3} \). The transformation \( \Theta = u^2 \) acts as a filter, as it requires \( u \) to be real and non-negative to produce a physically meaningful \( \Theta \). Consequently, the complex branches from the \( u \)-domain are entirely discarded. The resulting bifurcation diagram in \( \Theta \) is significantly simplified, showing a single, continuous, and smooth branch. This highlights that the apparent complexity in the algebraic \( u \)-space does not necessarily translate to the physical domain, emphasizing the importance of the transformation's constraints.

\begin{figure}[h]
    \centering
    \includegraphics[width=\textwidth]{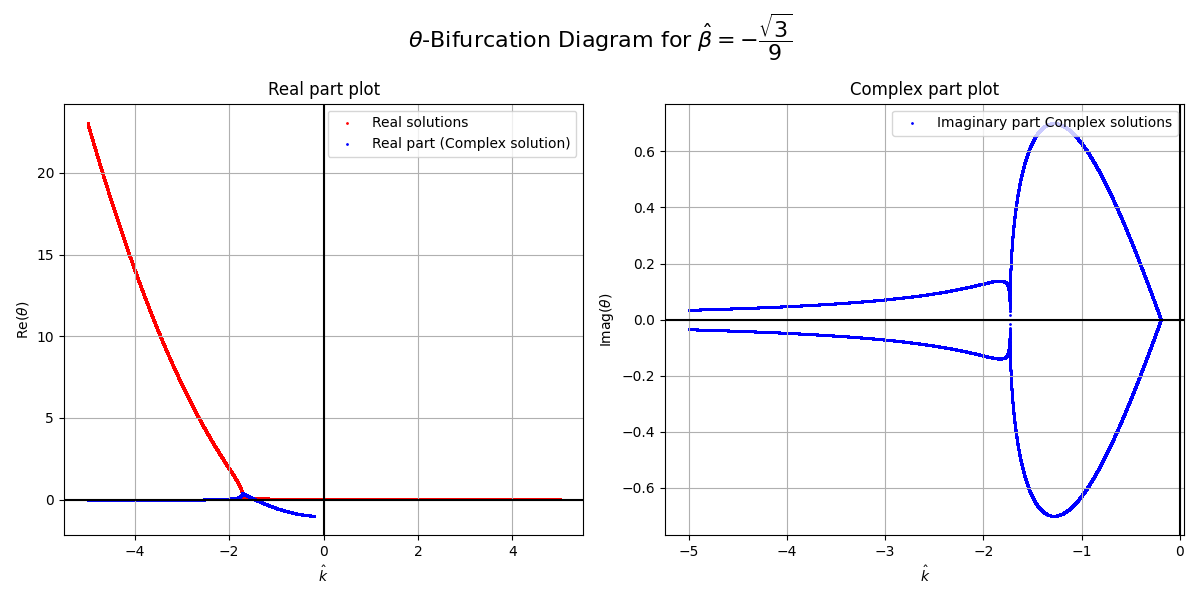}
    \caption{$\Theta$ solutions given $\hat{\beta}=-\tfrac{\sqrt{3}}{9}$.}
    \label{bif_theta_b_minus_sqrt_3_over_9}
\end{figure}

Figure \ref{bif_theta_b_minus_sqrt_3_over_9}, depicts the \( \Theta \)-solutions corresponding to the \( u \)-solutions of Figure \ref{bif_u_b_minus_sqrt_3_over_9}. The mapping process again simplifies the diagram by removing non-admissible complex solutions. However, a remnant of the loop structure from the \( u \)-domain is preserved, appearing here as a cusp or a sharp turn in the \( \Theta \) branch. This feature indicates a parameter region where the qualitative behavior of the physical system changes abruptly, corresponding to a bifurcation in the underlying effective potential. The diagram demonstrates how the insulated diode's state depends critically on the parameter \( \hat{k} \).

To complement these examples, we also consider the particular cases $\hat{k}=0$ and $\hat{\beta}=0$.

\begin{figure}[h]
    \centering
    \includegraphics[width=\textwidth]{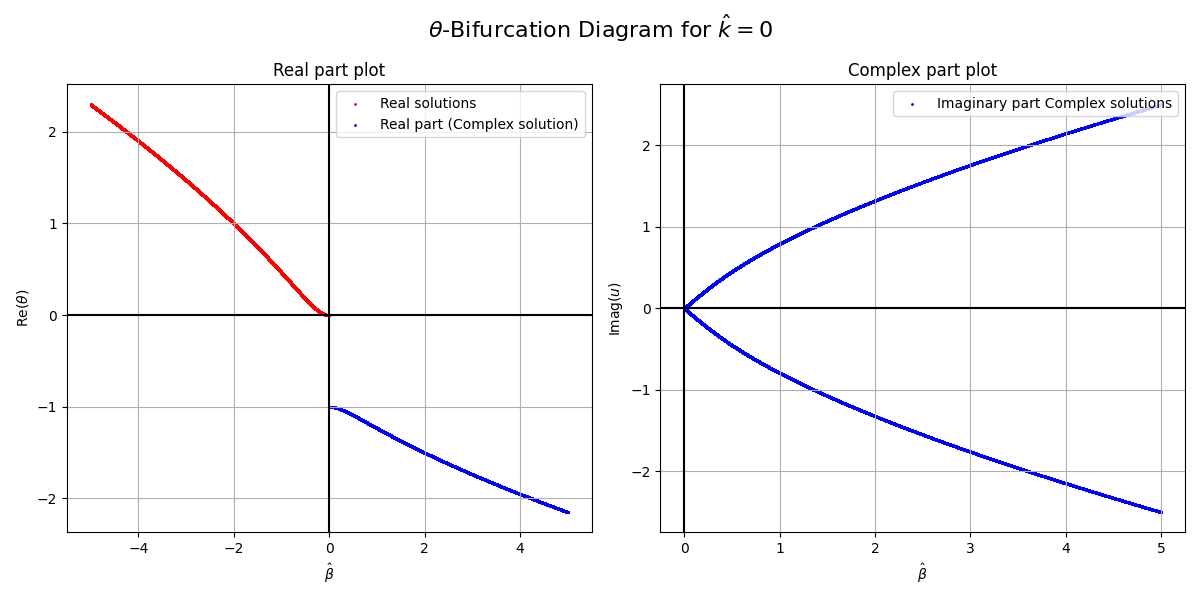}
    \caption{$\Theta$ solutions given $\hat{k}=0$.}
    \label{bif_theta_k_0}
\end{figure}

Figure \ref{bif_theta_k_0} presents the \( \Theta \)-bifurcation diagram for the case \( \hat{k}=0 \). The squaring transformation \( \Theta = u^2 \) merges the positive and negative parts of the real \( u \)-solutions from Figure \ref{bif_u_k_0}, resulting in a single, V-shaped curve. The complex loops from the \( u \)-domain have no valid mapping and are absent. This plot confirms that for this symmetric parameter choice, there is a unique, physically admissible value of the effective potential \( \Theta \) for each value of \( \hat{\beta} \), leading to a straightforward relationship devoid of bistability or bifurcations in this projection.

\begin{figure}[h]
    \centering
    \includegraphics[width=\textwidth]{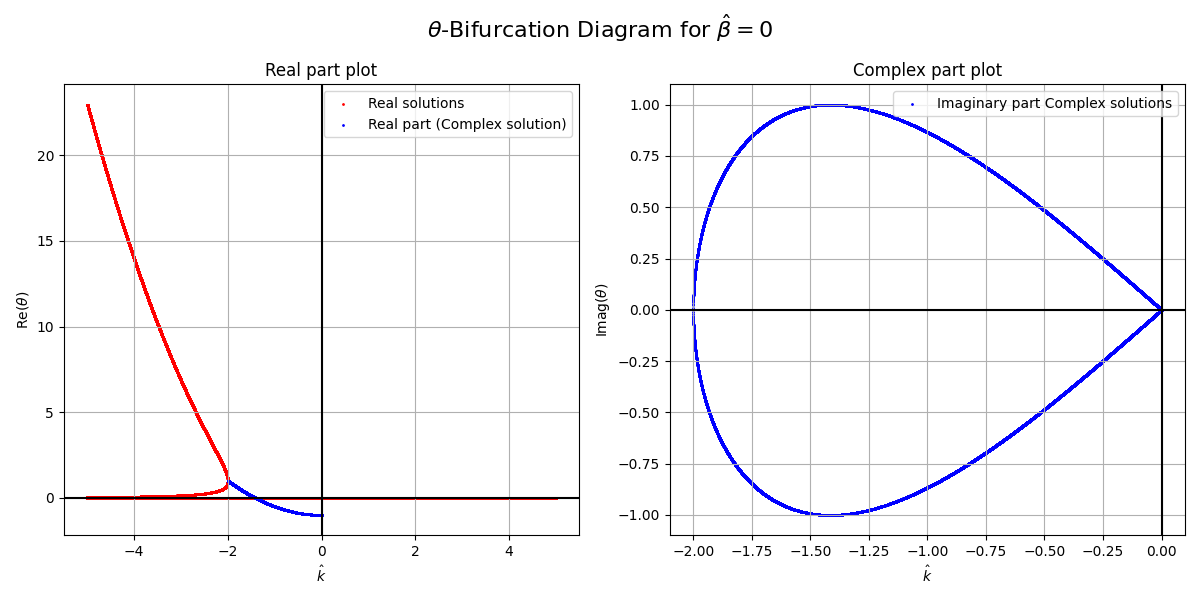}
    \caption{$\Theta$ solutions given $\hat{\beta}=0$.}
    \label{bif_theta_b_0}
\end{figure}

Figure \ref{bif_theta_b_0} shows the \( \Theta \)-solutions when \( \hat{\beta}=0 \). The resulting curve is a smooth, parabolic-like branch. The simplification from the corresponding \( u \)-plot (Figure \ref{bif_u_b_0}) is drastic; the intricate figure-eight loop has been completely filtered out by the \( \Theta = u^2 \) mapping, as it originated from complex-valued \( u \). This underscores that the physically observable states of the system are a subset of the full mathematical solution set, and the admissible solutions form a well-behaved, continuous family in this parameter cross-section.

These plots (Figures \ref{bif_theta_k_0} and \ref{bif_theta_b_0}) confirm that the $\Theta$ representation reduces the apparent complexity of the loops found in $u$ (Even disappearing in some cases) and allows us to visualize more clearly the ranges of feasible solutions. This reduction is due to the intrinsic constraints of the transformation $\Theta=\Theta(u)$, which discards non-admissible values.  

Finally, when we allow both parameters $\hat{k}$ and $\hat{\beta}$ to vary simultaneously, we obtain the global surface of solutions in $\Theta$.

\begin{figure}[h]
    \centering
    \includegraphics[width=\textwidth]{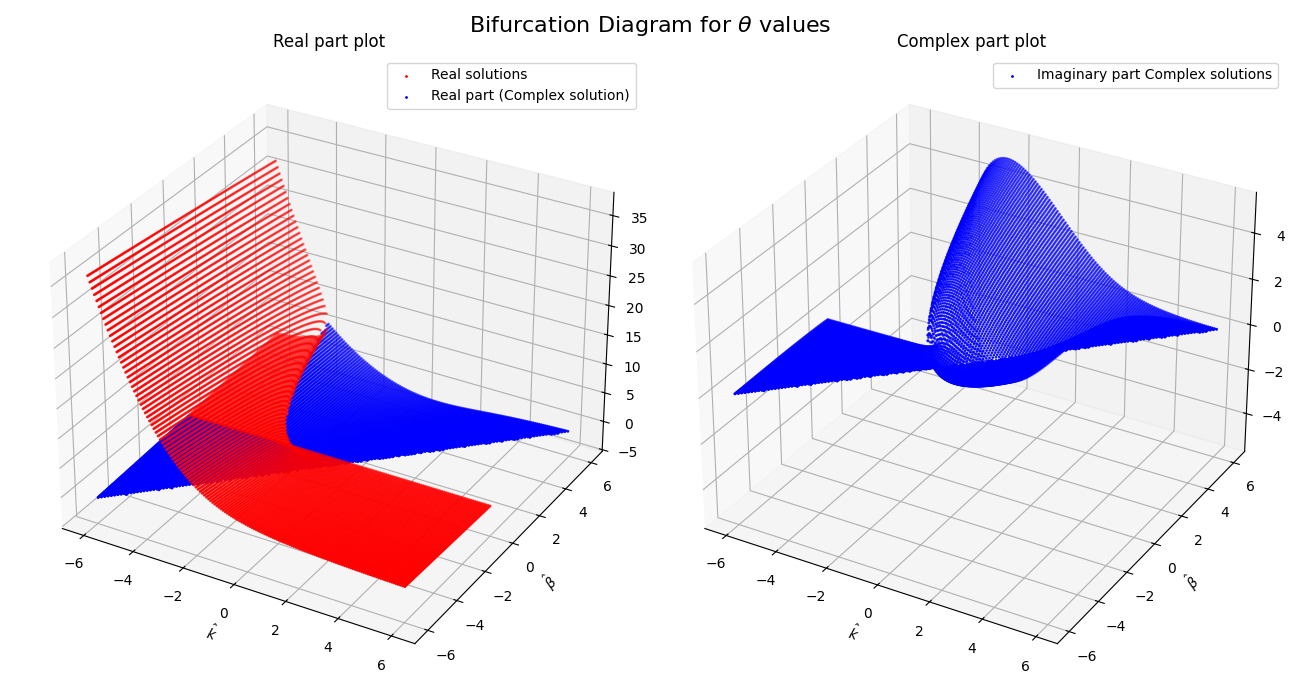}
    \caption{$\Theta$ solutions surface.}
    \label{bif_theta}
\end{figure}

Figure \ref{bif_theta} provides the global surface of all physically admissible \( \Theta \)-solutions. Compared to the complex, multi-sheeted surface in the \( u \)-domain (Figure \ref{bif_u}), this \( \Theta \)-surface is remarkably smoother and more compact. The folds and self-intersections have been largely eliminated by the mapping, which discards solutions with negative or complex \( u \). The resulting surface clearly delineates the region of parameter space \( (\hat{k}, \hat{\beta}) \) that supports a physical solution for the effective potential. The boundaries of this admissible region are of central physical importance, as they represent the critical thresholds for the onset of magnetic insulation.

\subsection{$\hat{k}$ , $\hat{\beta}$ influence over $\Theta_d$}

The admissible solutions for $\Theta_d$ are deeply linked with the relationship between the two parameters $\hat{k}$ and $\hat{\beta}$. 

We will first focus in the boundary that relates the complex solutions with imaginary part and the real solutions, which is the same as $\Delta_u = 0$
\\ \\
\begin{prop}\label{prop_bk_eq_0}
Let
\[
\Delta_u = 18 \hat{k} \hat{\beta} + \hat{k}^2 - 4 - 4 \hat{k}^3 \hat{\beta} - 27 \hat{\beta}^2.
\]
Then, the condition $\Delta_u = 0$ implies that $\hat{\beta}$ can be expressed in terms of $\hat{k}$ as
\begin{equation}
\hat{\beta} = \frac{\hat{k} \left( 9 - 2 \hat{k}^2 \pm 2 \sqrt{\hat{k}^4 - 9 \hat{k}^2 + 27/2} \right)}{27}.
\label{eq:beta_delta0}
\end{equation}
In other words, there are two branches:
\[
\hat{\beta}_+ = \frac{\hat{k} \left( 9 - 2 \hat{k}^2 + 2 \sqrt{\hat{k}^4 - 9 \hat{k}^2 + 27/2} \right)}{27}, \quad
\hat{\beta}_- = \frac{\hat{k} \left( 9 - 2 \hat{k}^2 - 2 \sqrt{\hat{k}^4 - 9 \hat{k}^2 + 27/2} \right)}{27}.
\]
\end{prop}
\begin{proof}
We start with the quadratic in $\hat{\beta}$ obtained from $\Delta_u = 0$:
\[
-27 \hat{\beta}^2 + (18 \hat{k} - 4 \hat{k}^3) \hat{\beta} + (\hat{k}^2 - 4) = 0.
\]

Multiply through by $-1$ to simplify:
\[
27 \hat{\beta}^2 + (4 \hat{k}^3 - 18 \hat{k}) \hat{\beta} + (4 - \hat{k}^2) = 0.
\]

Applying the quadratic formula gives:
\[
\hat{\beta} = \frac{-(4 \hat{k}^3 - 18 \hat{k}) \pm \sqrt{(4 \hat{k}^3 - 18 \hat{k})^2 - 108 (4 - \hat{k}^2)}}{54}.
\]

Simplifying the numerator and discriminant:

\[
-(4 \hat{k}^3 - 18 \hat{k}) = 18 \hat{k} - 4 \hat{k}^3 = 2 \hat{k} (9 - 2 \hat{k}^2),
\]
\[
(4 \hat{k}^3 - 18 \hat{k})^2 - 108 (4 - \hat{k}^2) = 16 \hat{k}^2 (\hat{k}^4 - 9 \hat{k}^2 + 27/2).
\]

Thus, the solutions become
\[
\hat{\beta} = \frac{2 \hat{k} (9 - 2 \hat{k}^2) \pm 4 \hat{k} \sqrt{\hat{k}^4 - 9 \hat{k}^2 + 27/2}}{54} = \frac{\hat{k} \left( 9 - 2 \hat{k}^2 \pm 2 \sqrt{\hat{k}^4 - 9 \hat{k}^2 + 27/2} \right)}{27}.
\]

This gives the two branches $\hat{\beta}_+$ and $\hat{\beta}_-$, as stated in \eqref{eq:beta_delta0}
\end{proof}

This analytical result can be effectively illustrated through computational visualization. By plotting the two branches of $\hat{\beta}$ as functions of $\hat{k}$, given by equation \eqref{eq:beta_delta0}, we can observe the locus of points where $\Delta_u = 0$. 

\begin{figure}[h]
    \centering
    \includegraphics[width=\textwidth]{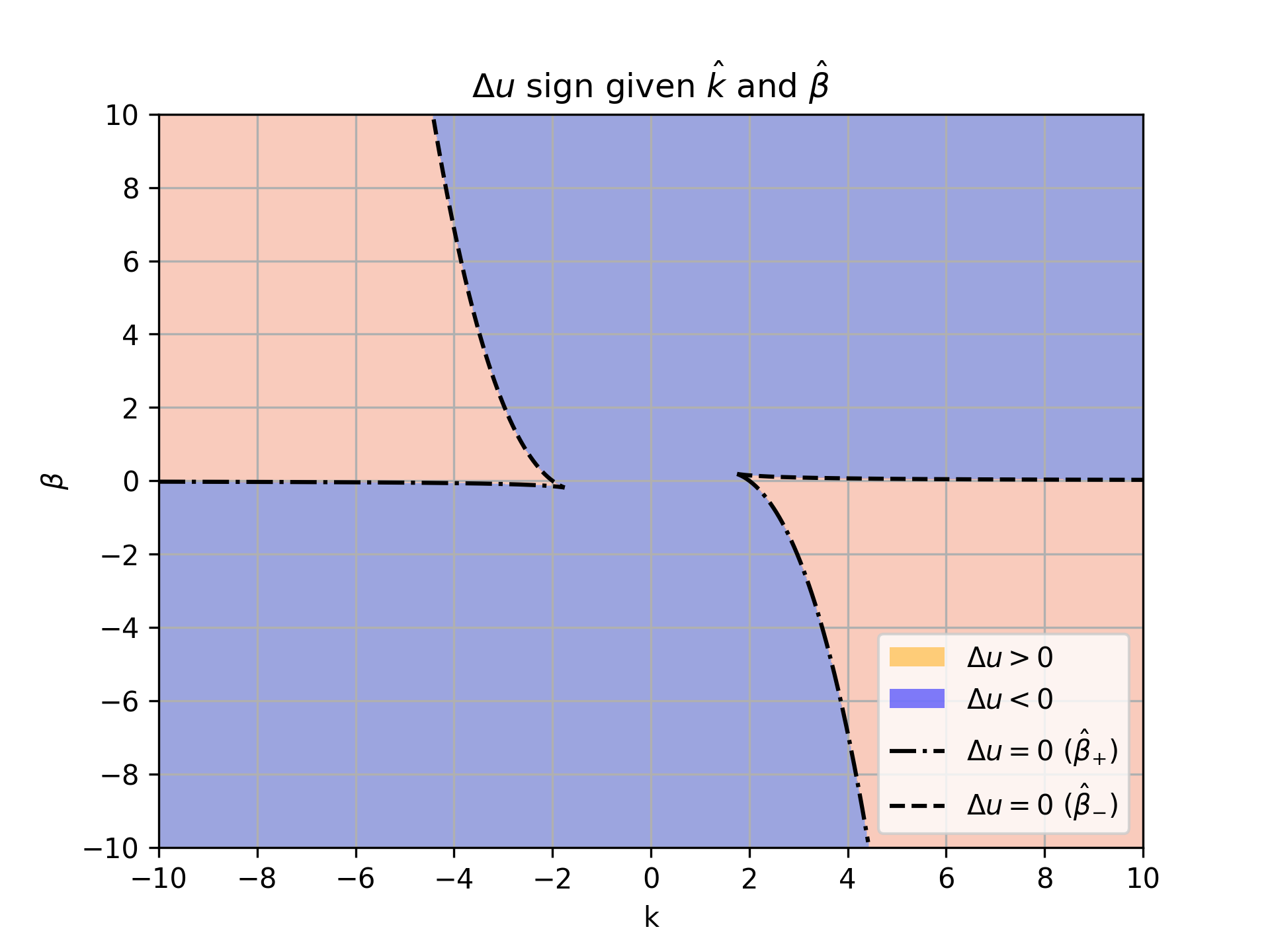}
    \caption{$\Delta_u$ sign boundaries}
    \label{deltau_signs}
\end{figure}

By plotting the two branches of $\hat{\beta}$ as functions of $\hat{k}$ given by equation (\ref{eq:beta_delta0}), we obtain the locus of critical points where the discriminant $\Delta_u = 0$. 
Figure \ref{deltau_signs} plots the analytical boundaries in the \( (\hat{k}, \hat{\beta}) \) parameter plane defined by \( \Delta_u = 0 \), as derived in Proposition 7. The two distinct branches, \( \hat{\beta}_+ \) and \( \hat{\beta}_- \), form a closed curve that separates the plane into two distinct regions. Inside this closed curve, the discriminant is positive (\( \Delta_u > 0 \)), indicating that the cubic equation admits three distinct real solutions for \( u \). Outside the curve, the discriminant is negative (\( \Delta_u < 0 \)), and there is only one real solution (and two complex conjugates). This diagram is fundamental as it maps the algebraic bifurcation set of the problem, predicting where qualitative changes in solution multiplicity occur. These curves therefore define the algebraic bifurcation set of the problem, and crossing them corresponds to a qualitative change in the multiplicity of solutions for the effective potential $\Theta(x)$.

Now we will focus on the scenario where we have the complex solutions, which is the same as $\Delta_u < 0$. In this scenario, we took into account the scenarios related to the real solution and the two conjugated solutions with imaginary parts.
\\ \\
\begin{prop}\label{prop_delta_min_zero} Let
\[
\begin{aligned}
\Delta_u &= 18\,\hat{k}\,\hat{\beta} + \hat{k}^2 - 4 - 4\,\hat{k}^3 \hat{\beta} - 27\, \hat{\beta}^2, \\
A_1 &= -\bigl(54\,\hat{k}^3 - 243\,\hat{k} + 729\,\hat{\beta}\bigr), \\
A_2 &= \sqrt{\bigl(54\,\hat{k}^3 - 243\,\hat{k} + 729\,\hat{\beta}\bigr)^2 + 2916\, (3 - \hat{k}^2)^3}.
\end{aligned}
\]

If $\Delta_u < 0 \qquad \text{and} \qquad \frac{6 \hat{k}}{\sqrt[3]{4}} < \sqrt[3]{A_1 + A_2} + \sqrt[3]{A_1 - A_2}$

then it follows that 
$\hat{\beta} < 0.$
\end{prop}
\begin{proof} We begin by defining the expression:
\[
S = \sqrt[3]{A_1 + A_2} + \sqrt[3]{A_1 - A_2},
\]
where $A_1$ and $A_2$ are functions of $\hat{k}$ and $\hat{\beta}$. From algebra, more specifically using Cardano's rule \cite{Cardano1545}, we know that $S$ can be written as the real root of a cubic equation of the form:
\[
S^3 + p S + q = 0.
\]

In our scenario, the cubic equation associated with $S$ is:
\begin{equation}
S^3 + 27 \sqrt[3]{4}\,(3 - \hat{k}^2)\,S + 54\,(2 \hat{k}^3 - 9 \hat{k} + 27 \hat{\beta}) = 0,
\label{eq:cubic_S}
\end{equation}
where $S$ is its unique real root when $\Delta_u < 0$.

Next, we introduce the candidate value:
\[
S_0 = \frac{6 \hat{k}}{\sqrt[3]{4}}.
\]

Since $\Delta_u < 0$, the cubic \eqref{eq:cubic_S} has a single real root and is monotonic along the real axis. This means that comparing $S_0$ with $S$ is equivalent to evaluating the sign of the cubic function at $S_0$:
\[
f(S_0) := S_0^3 + 27 \sqrt[3]{4}\,(3 - \hat{k}^2)\,S_0 + 54\,(2 \hat{k}^3 - 9 \hat{k} + 27 \hat{\beta}).
\]

Substituting $S_0$ into the cubic gives:
\[
\begin{aligned}
S_0^3 &= \left(\frac{6 \hat{k}}{\sqrt[3]{4}}\right)^3 = 54 \hat{k}^3,\\
27 \sqrt[3]{4}\,(3 - \hat{k}^2)\,S_0 &= 162 \hat{k}(3 - \hat{k}^2) = 486 \hat{k} - 162 \hat{k}^3,\\
54 (2 \hat{k}^3 - 9 \hat{k} + 27 \hat{\beta}) &= 108 \hat{k}^3 - 486 \hat{k} + 1458 \hat{\beta}.
\end{aligned}
\]

Adding these three terms, all contributions depending on $\hat{k}$ cancel, leaving:
\[
f(S_0) = 1458 \hat{\beta}.
\]

Because the cubic is monotonic, the condition $S_0 < S$ is equivalent to $f(S_0) < 0$, which immediately implies:
\[
1458 \hat{\beta} < 0 \quad \Longrightarrow \quad \hat{\beta} < 0.
\]

Hence, under the hypotheses $\Delta_u < 0$ and $S_0 < S$, we conclude that $\hat{\beta} < 0$.
\end{proof}

\begin{prop}\label{prop_delta_min_zero_k} Let
\[
\begin{aligned}
\Delta_u &= 18\,\hat{k}\,\hat{\beta} + \hat{k}^2 - 4 - 4\,\hat{k}^3 \hat{\beta} - 27\, \hat{\beta}^2, \\
A_1 &= -\left( 54\,\hat{k}^3 - 243\,\hat{k} + 729\,\hat{\beta} \right), \\
A_2 &= \sqrt{\left( 54\,\hat{k}^3 - 243\,\hat{k} + 729\,\hat{\beta} \right)^2 + 2916\, (3 - \hat{k}^2)^3}.
\end{aligned}
\]

Then, if $\Delta_u < 0$ and
\[
-\frac{12 \hat{k}}{\sqrt[3]{4}} > \sqrt[3]{A_1 + A_2} + \sqrt[3]{A_1 - A_2},
\]
we can conclude that $\hat{\beta} > \hat{k}.$
\end{prop}

\begin{proof} We define the same expression as before:
\[
S = \sqrt[3]{A_1 + A_2} + \sqrt[3]{A_1 - A_2},
\]
with
\[
\begin{aligned}
A_1 &= -\left( 54 \hat{k}^3 - 243 \hat{k} + 729 \hat{\beta} \right), \\
A_2 &= \sqrt{(54 \hat{k}^3 - 243 \hat{k} + 729 \hat{\beta})^2 + 2916 (3 - \hat{k}^2)^3}.
\end{aligned}
\]

Using proposition \ref{prop_delta_min_zero} strategy, we define $S$ is the unique real root of the cubic \eqref{eq:cubic_S}.

We now consider the candidate value:
\[
S_0 = -\frac{12 \hat{k}}{\sqrt[3]{4}}.
\]

Since $\Delta_u < 0$, the cubic \eqref{eq:cubic_S} is monotonic along the real axis. Therefore, the inequality $S_0 > S$ is equivalent to evaluating the cubic at $S_0$:
\[
f(S_0) := S_0^3 + 27 \sqrt[3]{4}\,(3 - \hat{k}^2)\,S_0 + 54\,(2 \hat{k}^3 - 9 \hat{k} + 27 \hat{\beta}).
\]

Substituting $S_0$ into the cubic gives:
\[
\begin{aligned}
S_0^3 &= \left(-\frac{12 \hat{k}}{\sqrt[3]{4}}\right)^3 = -432 \hat{k}^3,\\
27 \sqrt[3]{4}\,(3 - \hat{k}^2)\,S_0 &= 27 \sqrt[3]{4}\,(3 - \hat{k}^2)\, \left(-\frac{12 \hat{k}}{\sqrt[3]{4}}\right) = -324 \hat{k} (3 - \hat{k}^2) = -972 \hat{k} + 324 \hat{k}^3,\\
54 (2 \hat{k}^3 - 9 \hat{k} + 27 \hat{\beta}) &= 108 \hat{k}^3 - 486 \hat{k} + 1458 \hat{\beta}.
\end{aligned}
\]

Adding these three contributions:
\[
f(S_0) = (-432 + 324 + 108) \hat{k}^3 + (-972 - 486) \hat{k} + 1458 \hat{\beta} = 0 - 1458 \hat{k} + 1458 \hat{\beta}.
\]

We see that now there is a \textbf{combined contribution of $\hat{k}$ and $\hat{\beta}$}. For the inequality $S_0 > S$ to hold, we require $f(S_0) > 0$, which gives the relation:
\[
1458 (\hat{\beta} - \hat{k}) > 0 \quad \Longrightarrow \quad \hat{\beta} > \hat{k}.
\]

Thus, under the hypotheses $\Delta_u < 0$ and $S_0 > S$, we conclude that $\hat{\beta} > \hat{k}$.
\end{proof}

These two particular new sets of rules defined in proposition \ref{prop_delta_min_zero} and \ref{prop_delta_min_zero_k} lead to new boundary definitions that can be verified via computer simulation.

\begin{figure}[h]
    \centering
    \includegraphics[width=\textwidth]{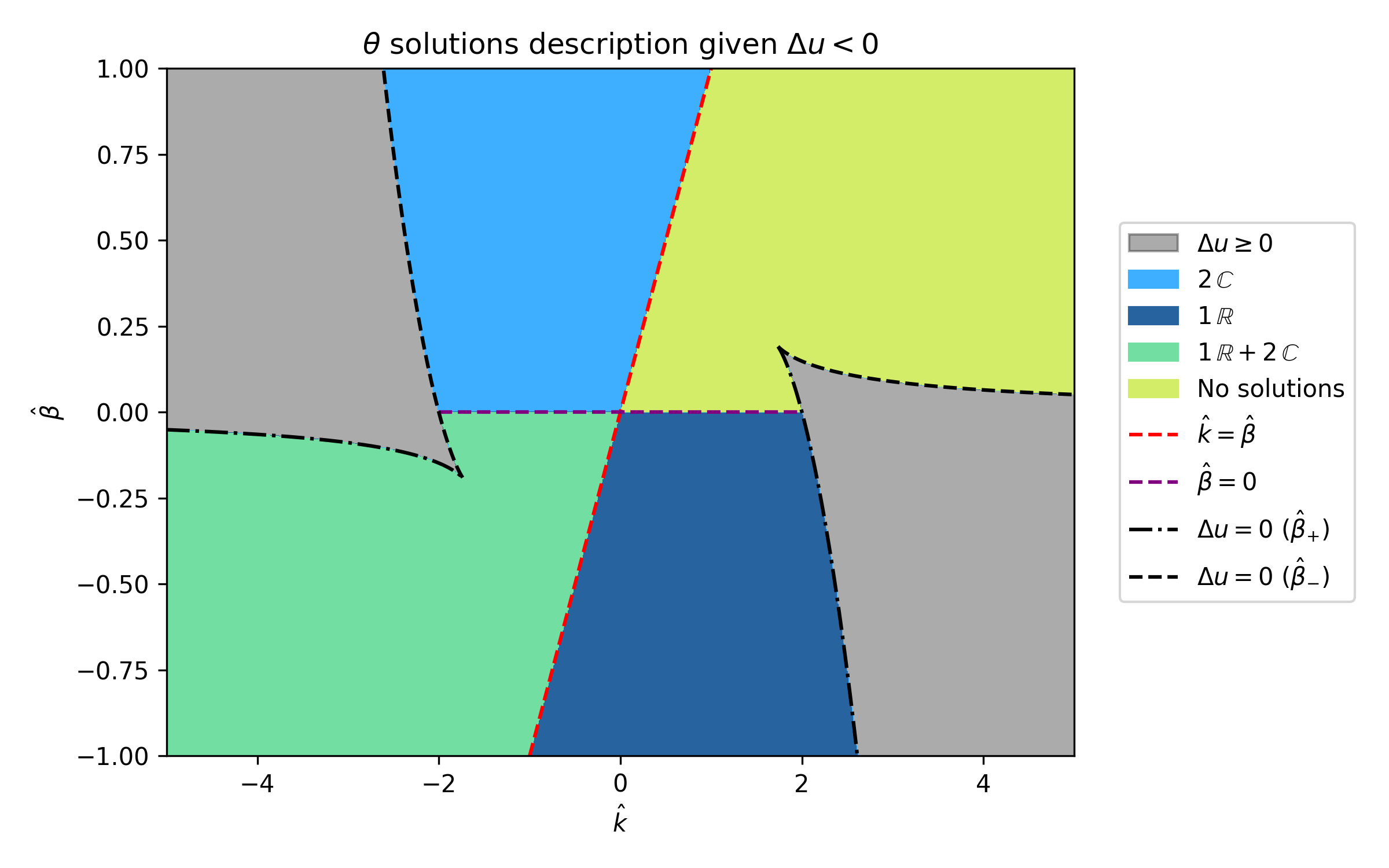}
    \caption{$\Theta_d$ solutions given $\Delta_u < 0$ }
    \label{deltau_signs_neg}
\end{figure}

Figure \ref{deltau_signs_neg} refines the algebraic picture of Figure \ref{deltau_signs} by incorporating the physical constraint \( \Theta \geq 0 \). It shows the distribution of the \textit{physically admissible} \( \Theta_d \) solutions in the parameter plane for the case \( \Delta_u < 0 \). While the algebraic discriminant indicates one real \( u \)-solution in this region, this solution must also yield a non-negative \( \Theta_d \) to be physically valid. The plot reveals that this additional constraint further trims the solution set, carving out specific sub-regions (likely shown with distinct colors or patterns) where a physically insulated diode regime exists. This highlights that the operational regime of the diode is determined by a combination of algebraic conditions and physical realizability.

This highlights an important feature of the system: bifurcations are not purely algebraic but are also constrained by the physical requirement that the effective potential remain nonnegative. 
The admissible regions in the diagram reveal how parameter variations in $j_x$ and $\beta$ determine the transition of the diode between non-insulated and insulated regimes.

In order to finish this section, we will study the parameters' impact over the solutions in the area where $\Delta_u > 0$. For this, we analyzed the three scenarios of real solutions.

\begin{prop}\label{prop_delta_over_zero} Let
\[
\begin{aligned}
\Delta_u &= 18\,\hat{k}\,\hat{\beta} + \hat{k}^2 - 4 - 4\,\hat{k}^3 \hat{\beta} - 27\, \hat{\beta}^2, \\
U &= \dfrac{\hat{k}}{2 \sqrt{\hat{k}^2 -3}} \\
V &= \dfrac{2\hat{k}^3 - 9\hat{k} + 27 \hat{\beta} }{18-6\hat{k}^2 } \sqrt{\dfrac{9}{\hat{k}^2 -3}}
\end{aligned} 
\]

If $\Delta_u > 0 \quad \text{and} \quad \arccos(U) =   \dfrac{1}{3}\arccos(V) $

then it follows $\hat{\beta} = 0$ with $\hat{k}^2 \geq 4$
\end{prop}

\begin{proof}
 We begin with expression $\arccos(U) = \dfrac{1}{3}\arccos(V)$. We rewrite it in order to use the triple angle identities and define a relationship between $U$ and $V$ without trigonometric functions.

\[
\begin{aligned}
    \arccos(U) &=   \dfrac{1}{3}\arccos(V)  \\
    3 \arccos(U) &=  \arccos(V)  \\
    \cos(3\arccos(U)) &= V \\
    4 \cos^3(\arccos(U)) - 3 \cos(\arccos(U)) &= V \\
    4U^3 - 3U &= V
\end{aligned} 
\]

Using this mid result, we will use algebraic manipulation in order to rewrite the expression for $\hat{\beta}$ and $\hat{k}$

Using the following notation for $U$ and $V$

$$V = \dfrac{2\hat{k}^3 - 9\hat{k} + 27 \hat{\beta} }{18-6\hat{k}^2 } \sqrt{\dfrac{9}{\hat{k}^2 -3}} = \dfrac{2\hat{k}^3 - 9\hat{k} + 27 \hat{\beta} }{-2(\hat{k}^2 - 3)^{3/2} }$$

$$U = \dfrac{\hat{k}}{2 \sqrt{\hat{k}^2 -3}} = \dfrac{\hat{k}}{2 (\hat{k}^2 -3)^{1/2}}$$

We get that

\[
\begin{aligned}
    4U^3 - 3U &= V \\
    4\left(\dfrac{\hat{k}}{2 (\hat{k}^2 -3)^{1/2}}\right)^3 - 3\left(\dfrac{\hat{k}}{2 (\hat{k}^2 -3)^{1/2}}\right) &= \dfrac{2\hat{k}^3 - 9\hat{k} + 27 \hat{\beta} }{-2(\hat{k}^2 - 3)^{3/2} } \\
   \dfrac{\hat{k}^3}{2 (\hat{k}^2 -3)^{3/2}} - \dfrac{3\hat{k}}{2 (\hat{k}^2 -3)^{1/2}} &= \dfrac{2\hat{k}^3 - 9\hat{k} + 27 \hat{\beta} }{-2(\hat{k}^2 - 3)^{3/2} } \\
   \dfrac{\hat{k}^3}{2 (\hat{k}^2 -3)^{3/2}} + \dfrac{2\hat{k}^3 - 9\hat{k} + 27 \hat{\beta} }{2(\hat{k}^2 - 3)^{3/2} } - \dfrac{3\hat{k}}{2 (\hat{k}^2 -3)^{1/2}} &= 0 \\
   \dfrac{27 \hat{\beta} }{2(\hat{k}^2 - 3)^{3/2} } &= 0 \\
\end{aligned} 
\]

This last expression leads to the following result:

\[
\frac{27\hat{\beta}}{2(\hat{k}^2-3)^{3/2}} = 0
\Longleftrightarrow
\hat{k}^2\geq 3 \quad\wedge\quad \hat{\beta}=0
\]

In addition, we must take into account the boundary given by the $\arccos$ function.

\[
\begin{alignedat}{2}
    |U| &\leq 1 & \quad |V| &\leq 1 \\
    U^2 &\leq 1 & \quad V^2 &\leq 1 \\
    \dfrac{\hat{k}^2}{4 (\hat{k}^2 -3)} &\leq 1 & \quad 
    \dfrac{(2\hat{k}^3 - 9\hat{k} + 27 \hat{\beta})^2 }{4(\hat{k}^2 - 3)^{3} } &\leq 1 \\
    \hat{k}^2 &\leq 4 \hat{k}^2 - 12 & \quad
    (2\hat{k}^3 - 9\hat{k} + 27 \hat{\beta})^2 &\leq 4(\hat{k}^2 - 3)^{3} \\
    12&\leq 3\hat{k}^2 & \quad
    -\hat{k}^2 + 4\hat{k}^3\hat{\beta} - 18\hat{k}\hat{\beta} + 27\hat{\beta}^2 + 4 &\leq 0 \\
    4&\leq \hat{k}^2 & \quad
    -\Delta_u &\leq 0 \\
    & & \quad \Delta_u &\geq 0
\end{alignedat}
\]

Combining both results, we obtain
\[
\hat{k}^2 \geq 3 \quad\wedge\quad \hat{\beta}=0,
\]
and from the boundary conditions of the $\arccos$ function together with $\Delta_u>0$, it follows that in fact
\[
\hat{k}^2 \geq 4 \quad\wedge\quad \hat{\beta}=0.
\]

Therefore, under the assumptions $\Delta_u>0$ and $\arccos(U)=\tfrac{1}{3}\arccos(V)$, the only possible case is
\[
\hat{\beta}=0 \quad\text{with}\quad \hat{k}^2 \geq 4.
\]
\end{proof}

The boundary defined in Proposition \ref{prop_delta_over_zero} can be clearly seen using numerical simulation.

\begin{figure}[h]
    \centering
    \includegraphics[width=\textwidth]{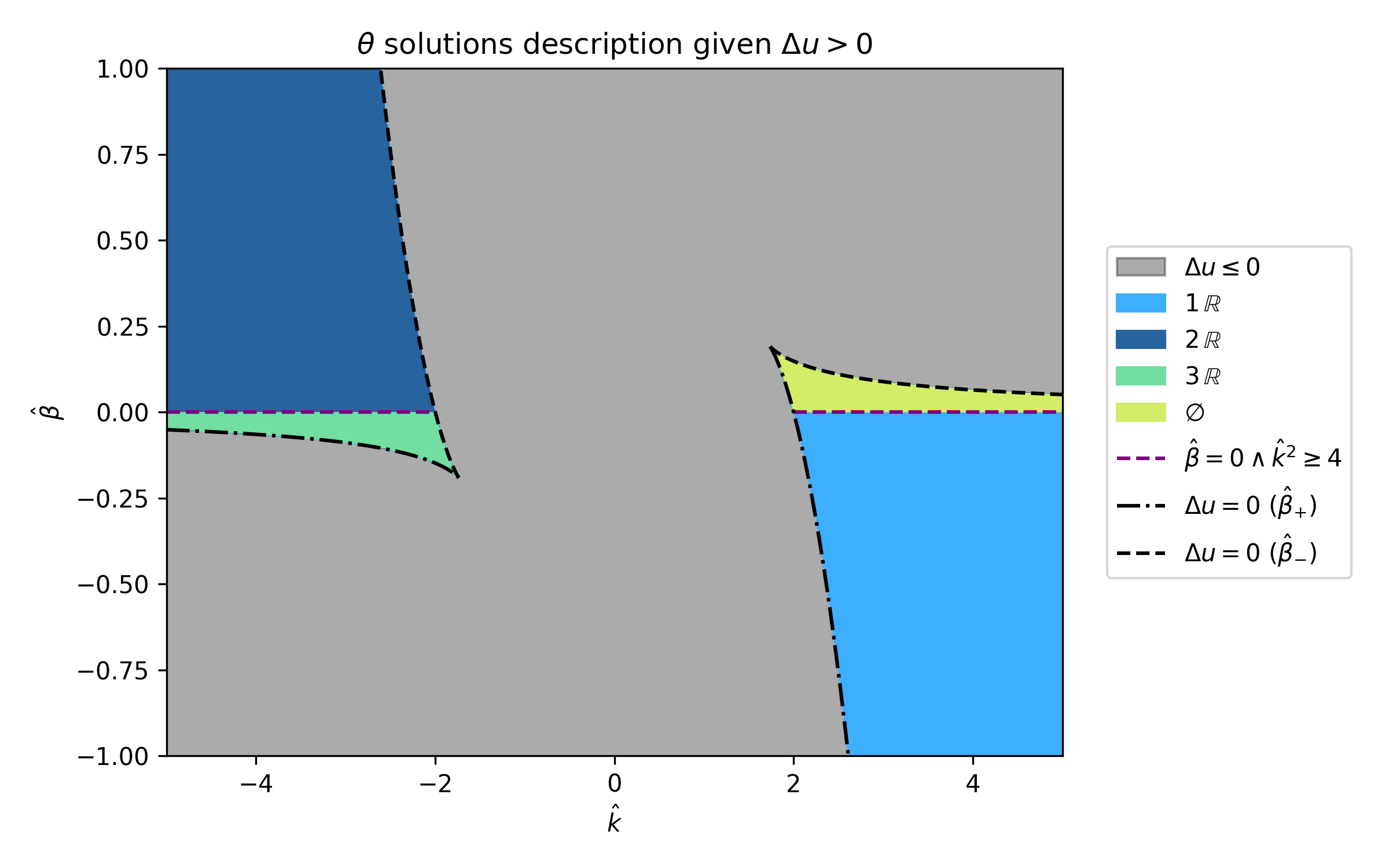}
    \caption{$\Theta_d$ solutions given $\Delta_u < 0$ }
    \label{deltau_signs_plus}
\end{figure}

Figure \ref{deltau_signs_plus} provides a detailed numerical map of the \( \Theta_d \) solutions in the parameter space, specifically focusing on the region where \( \Delta_u > 0 \) (three real \( u \)-solutions). The plot clearly shows boundaries that partition the region into sub-domains, each corresponding to a different number of \textit{physically admissible} \( \Theta_d \) solutions (i.e., those where \( u \) is real and results in \( \Theta_d \geq 0 \)). The boundaries derived from the trigonometric condition in Proposition 10 are visible and validate the analytical work. This figure is crucial for understanding the complex multistability and bifurcation phenomena that can occur in the insulated diode regime, showing how multiple physical states can coexist for the same control parameters.

\begin{figure}[h]
    \centering
    \includegraphics[width=\textwidth]{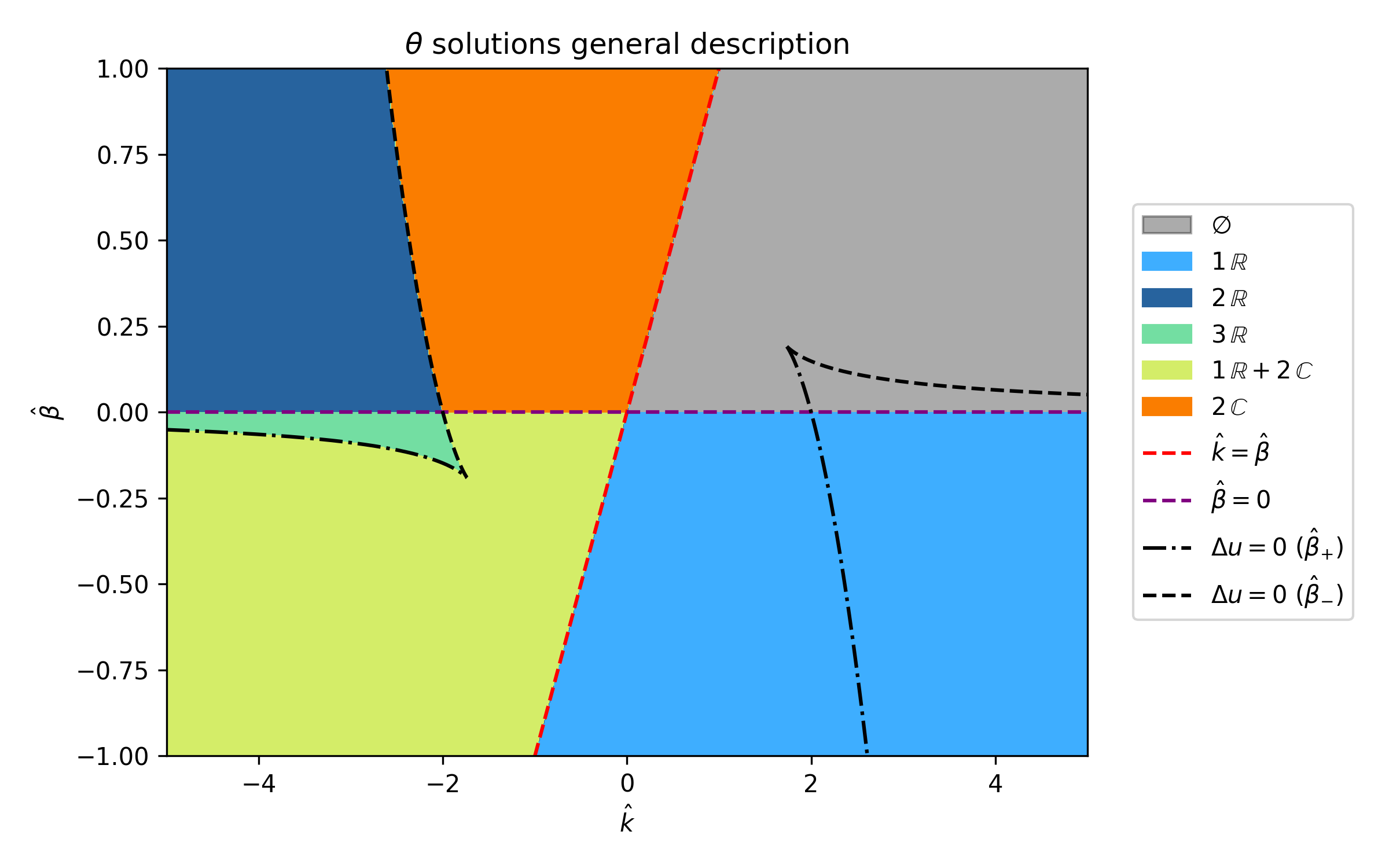}
    \caption{$\Theta_d$ solutions given $\Delta_u < 0$ }
    \label{deltau_signs_full}
\end{figure}

Figure \ref{deltau_signs_full} synthesizes the analytical and numerical results into a complete phase diagram of the solution structure for the effective potential \( \Theta_d \). It overlays the boundaries from Propositions \ref{prop_delta_min_zero} , \ref{prop_delta_min_zero_k} and \ref{prop_delta_over_zero} onto the parameter plane, creating a comprehensive chart. This final diagram labels distinct regions such as ``1 Real Adm. Solution,'' ``3 Real Adm. Solutions,'' and ``1 Real + 2 Complex,'' providing an immediate understanding of the solution landscape for any given pair \( (\hat{k}, \hat{\beta}) \). It serves as a master plot for predicting the diode's behavior, illustrating the direct link between the abstract parameters of the model and the physically observable states of magnetic insulation.

\section{Conclusions}

For practitioners in high-power electronics, this research offers a transformative toolkit for designing magnetically insulated diodes. The study's primary value lies in replacing historical trial-and-error methods with a predictive, model-based framework. Engineers can now leverage the derived bifurcation diagrams as a direct design guide to select parameters that ensure stable diode insulation. Similarly, the calculation of the 'insulated diode spacing' provides a concrete criterion for a fundamental design variable. By furnishing these rigorous tools - a mathematical model and a numerical analysis of bifurcation behavior - this work empowers the creation of significantly more efficient and reliable components for the most demanding power conversion and pulsed power applications.

In this work we have studied the boundary value problem of magnetically insulated diode, focusing on the analysis of the limit Cauchy problem and also on the isolated case. The first important step was to show that, under certain conditions, the reduced nonlinear system admits nonnegative solutions of the effective potential in the interval $[0, x^{*})$. These solutions coincide with the non--insulated regime and in fact they are compatible with the well known Child--Langmuir law [Langmuir, Compton, 1931] when the magnetic field vanishes. This makes a clear connection between the classical theory of space charge limitation and the more complicated case with magnetic insulation.  

For the insulated situation, we rewrote the conditions on the effective potential in terms of a cubic equation, which depends directly on the boundary conditions and on the parameters of the diode. The analysis of this equation allowed us to classify the different types of solutions depending on the sign of the discriminant, and in this way we obtained a description of the bifurcation structure of the problem. This classification gave us a better picture of how the transition between regimes occurs, how the critical values of the current density appear, and what role the free boundary $x^{*}$ plays in the formation of the electron layer near to the cathode.  

Even if the system is mathematically very demanding, we managed to combine analytical arguments with algebraic methods and numerical displays to obtain a coherent vision of the phenomena. In particular, the graphical exploration of bifurcations and of the effective potential supported the theoretical results and made them easier to interpret.  

Altogether, the results provide for the first time a unified description, both analytical and numerical, of the bifurcation phenomena in magnetically insulated diodes. They offer new insight on how magnetic insulation can control electron flow and energy transfer in vacuum devices. We believe that these findings open interesting possibilities for future research, especially in the direction of stability analysis and in possible applications in plasma physics and in high power electronic systems.

\section*{Acknowledgments}

The authors are grateful to Jurgen Batt, Pierre Degond and Walter Strauss for stimulating discussions. This work was partly funded 
by the Ministry of Science and Higher Education of Russian Federation, project FZZS-2024-0003 of INRTU.


\end{document}